\documentclass[twoside]{article}
\usepackage[dvips]{curves}
\usepackage{xcolor}
\usepackage{amsmath}
\usepackage{float}
\usepackage{graphicx}
\usepackage{amsfonts}
\usepackage{fancyhdr}
\usepackage[latin1]{inputenc}
\usepackage[T1]{fontenc}

\numberwithin{equation}{section}
\newtheorem{theorem}{Theorem}[section]

\newtheorem{remark}[theorem]{Remark}
\newtheorem{lemma}[theorem]{Lemma}

\newenvironment{proof}[1][Proof]{\textbf{#1} }{\ \rule{0.0em}{0.0em}}
\begin{document}

\title{  Quasilinear problems involving \\ a perturbation  with quadratic
growth in the gradient \\ and a noncoercive zeroth order term
}
\date{}
\maketitle \vspace{-1.5cm}
\begin{center}\author{Boussad Hamour \footnote{D\'{e}partement de Math\'{e}matiques, Ecole Normale
Sup\'erieure, Bo\^ite postale 92, Vieux Kouba, 16050 Alger,\break
Alg\'{e}rie; 
e-mail: hamour@ens-kouba.dz}
\quad\&\quad
Fran\c{c}ois Murat  \footnote{Laboratoire Jacques-Louis Lions,
Universit\'e Pierre et Marie Curie (Paris VI) et CNRS, 
Bo\^{\i}te courrier 187,\break 
75252 Paris Cedex 05,
 France; e-mail: murat@ann.jussieu.fr}}
\end{center}

\bigskip

\begin{abstract}
In this paper we consider the problem
\begin{equation*}
\left\{\begin{array}{ll} u\in H_{0}^{1}(\Omega),
\\ &
\\
-\textrm{div}\,(A(x)Du)=H(x,u,Du)+f(x)+a_{0}(x)\, u& \textrm{in}
\;\;\;\mathcal{D}'(\Omega),
\end{array}
\right.
\end{equation*}
where $\Omega$ is an open bounded set of $\mathbb{R}^{N}$, $N \geq
3$, $A(x)$ is a coercive matrix with coefficients in
$L^\infty(\Omega)$,  $H(x,s,\xi)$ is a Carath\'eodory function which
satisfies for some $\gamma >0$
$$
 -c_{0}\, A(x)\, \xi\xi\leq H(x,s,\xi)\,{\rm sign}(s)\leq \gamma\,A(x)\,\xi\xi \;\;\;
{\rm a.e. }\; x\in \Omega,\;\;\;\forall s\in\mathbb{R},\;\;\;
 \forall\xi \in \mathbb{R}^{N},
$$
 $f$ belongs to $L^{N/2}(\Omega)$  and  $a_{0} \geq 0$ to $L^{q}(\Omega )$, $ q>N/2 $.
For $f$  and   $a_{0}$  sufficiently small, we prove
the existence   of at least one solution $u$ of this problem which
is moreover such that $e^{\delta_0 |u|} -1$ belongs to $H_{0}^{1}(\Omega)$ for
some $\delta_0\geq\gamma$ and satisfies an a priori estimate.
\end{abstract}
\renewcommand{\abstractname}{R\'esum\'e}
\begin{abstract}
Dans cet article nous \'{e}tudions le probl\`{e}me
\begin{equation*}
\left\{\begin{array}{ll} u\in H_{0}^{1}(\Omega),
\\ &
\\
-\textrm{div}\,(A(x)Du)=H(x,u,Du)+f(x)+a_{0}(x)\, u& \textrm{dans
}\;\;\; \mathcal{D}'(\Omega),
\end{array}
\right.
\end{equation*}
o\`{u} $\Omega$ est un ouvert born\'{e} de $
\mathbb{R}^{N}$, $N\geq 3$, $A(x)$ est
 une matrice coercive \`{a} coefficients  $L^\infty(\Omega)$,
 $ H(x,s,\xi)$
 est une fonction de Carath\'edory qui satisfait pour un certain $\gamma >0$
\begin{equation*}
 -c_{0}\, A(x)\, \xi\xi\leq H(x,s,\xi)\,{\rm sign}(s)\leq \gamma\,A(x)\,\xi\xi \;\;\;
{\rm p.p. }\; x\in \Omega,\;\;\;\forall s\in\mathbb{R},\;\;\;
 \forall\xi \in \mathbb{R}^{N},
\end{equation*}
$f$ appartient \`a $L^{N/2}(\Omega)$ et $a_{0} \geq 0$ \`a $L^{q}(\Omega )$, $ q>N/2 $.
Pour  $f$ et $a_{0}$ suffisamment petits,
nous d\'{e}montrons qu'il existe au moins une solution $u$ de ce
probl\`{e}me qui est de plus telle que
$e^{\delta_0 |u|} -1$ appartient \`a $H_{0}^{1}(\Omega)$ pour un certain
$\delta_0\geq\gamma$ et satisfait une estimation a priori.
\end{abstract}

\vfill 
\eject


\section{Introduction}

In this paper, we consider the quasilinear problem
\begin{equation}
\left\{\begin{array}{ll}\label{02.1} u\in H_{0}^{1}(\Omega),
\\ 
\\
-\textrm{div}\,(A(x)Du)=H(x,u,Du)+f(x)+a_{0}(x)\, u 
\;\;\; \textrm{in}\;\;\;\;\mathcal{D}'(\Omega),
\end{array}
\right.
\end{equation}
where $\Omega$ is a bounded open set of $ \mathbb{R^{N}}$, $N \geq
3$, where $A$ is a coercive matrix with bounded measurable
coefficients,  where  $H(x,s,\xi)$ is a Carath\'{e}odory function
wich  has quadratic growth in $\xi$, and more precisely  which
satisfies for some $\gamma >0$ and $c_0\geq 0$
\begin{equation}\label{0H}
 -c_{0}\, A(x)\, \xi\xi\leq H(x,s,\xi)\,{\rm sign}(s)\leq \gamma\,A(x)\,\xi\xi ,\;\;\;
{\rm a.e. }\; x\in \Omega,\;\;\;\forall s\in\mathbb{R},\;\;\;
 \forall\xi \in \mathbb{R}^{N},
\end{equation}
 where $f\in L^{N/2}(\Omega)$, $f \neq 0$,
 and where $a_{0}\in L^{q}(\Omega)$, $q>\frac{N}{2}$,  with
\begin{equation}\label{01.3}
a_{0}\geq 0, \quad a_0\neq 0.
\end{equation}

\par When $ f$ and  $a_0$ are sufficiently small (and more precisely when $f$ and $a_0$  satisfy 
the two smallness conditions (\ref{A1}) and  (\ref{A3}), we prove in the present
paper that  problem (\ref{02.1}) has a least one solution, which is
moreover such that

 \begin{equation}\label{01.4}
{e^{\delta_0|u|}-1}\in H_0^1(\Omega) ,
\end{equation}
with
\begin{equation}\label{01.55}
\left\|\frac{e^{\delta_0|u|}-1}{\delta_0}\right\|_{H^{1}_{0}(\Omega)}\leq
Z_{\delta_{0}},
 \end{equation}
where $\delta_0\geq\gamma$  and $Z_{\delta_0}$ are two constants
which depend only on the data of the problem (see (\ref{3777bis}),
(\ref{Phi delta0 bis}), (\ref{6.15bis})  for the definitions of
$\delta_0$ and $Z_{\delta_0}$).
\par  The main originality of our result is the fact that  we assume that $a_0$ satisfies (\ref{01.3}), namely that $a_0$ is a
nonnegative function. 
\vspace{0.8cm}
\par Let us begin with some review of the literature.
\par Problem  (\ref{02.1}) has been studied in many papers in the case where  $a_0\leq 0$.  Among these papers is  a series of papers \cite{bmp3},  \cite{bmp4},  \cite{bmp2}
and   \cite{bmp6} by L.~Boccardo, F.~Murat and J.-P.~Puel (see also
the paper \cite{rakotoson} by J.-M.~Rokotoson),
 which are concerned with the case where
\begin{equation}\label{01.5}
a_{0}(x)\leq -\alpha_ 0<0.
\end{equation}
In  these papers  (which also consider nonlinear monotone
operators and not only the linear operator 
 $-{\rm div}\,(A(x)Du)$,  the authors prove that  when $a_0$
satisfies  (\ref{01.5}) and when $f$ belongs to $L^q(\Omega)$, $
q>\frac{N}{2}$,  then there exists at least one solution of
(\ref{02.1}) which moreover belongs to $L^\infty(\Omega)$ and which
satisfies some a priori estimates. The uniqueness of such a solution
has been proved, under some further structure  assumptions, by G.
Barles and F.~Murat  in \cite{barles},  by G. Barles, A.-P.~Blanc,
C.~ Georgelin and M. Kobylanski in \cite{barles 2} and by
 G. Barles  and A. Porretta in \cite{barles porretta}.

 \vfill
 \eject

\par The case where
\begin{equation}\label{01.31}
a_{0}= 0
\end{equation}
was considered,  among others,  by A. Alvino, P.-L. Lions and G.
Trombetti in \cite{alvino}, by  C. Maderna, C.~Pagani and S. Salsa
in \cite{salsa}, by V. Ferone and M.-R.  Posteraro in
\cite{feroneposteraro1}, and   by N. Grenon-Isselkou and  J. Mossino
in \cite{grenon}. In these papers (which also consider nonlinear
monotone operators), the authors prove that when $a_0$  satisfies
(\ref{01.31}) and when $f$ belongs to $L^q(\Omega)$,
$q>\frac{N}{2}$,   with  $ \|f\|_{L^q (\Omega )}$  sufficiently
small, then there exists at least one solution of (\ref{02.1}) which
moreover belongs to $L^\infty(\Omega)$ and which satisfies some a
priori estimates.
\par The case where $a_0$ satisfies  (\ref{01.31})  but where $f$ only belongs to $L^{N/2}(\Omega)$ for $N\geq 3$
(and no more to $L^q(\Omega)$ with  $q>\frac{N}{2}$)  was considered
by V. Ferone and F.~Murat in \cite{fm} (and  in  \cite{fm1} in the
nonlinear monotone case). These authors  proved that when $
\|f\|_{L^{N/2} (\Omega )}$ is sufficiently  small, there exists at
least one solution of (\ref{02.1}) which is moreover such that
$e^{\delta| u|}-1\in H_0^1(\Omega)$  for some $\delta>\gamma$,
and that such a solution satisfies 
an a priori estimate. Similar results were
obtained in the case where $f\in L^{N/2}(\Omega)$ by A.~Dall'Aglio,
D.~Giachetti and J.-P.~Puel in \cite{agp} for possibly unbounded domains when $a_0$ satisfies (\ref{01.5});
in this case no smallness condition is required on $f$.
Finally, in~\cite{fm2}, V.~Ferone and F.~Murat  considered (also in the case of nonlinear monotone operators) the case where $a_0$ satisfies $a_0\leq 0$ and where $f$ belongs to the Lorentz space $L^{N/2, \infty} (\Omega)$; in this case two smallness conditions should be fulfilled.
\par To finish with the case where $a_0$
satisfies $a_0\leq 0$, let us quote the paper \cite{porretta} by
A.~Porretta, where the author studies the asymptotic behaviour of
the solution $ u$  of (\ref{02.1}) when $a_0$ is a strictly positive
constant which tends to zero,  and proves that an ergodic constant
appears at the limit $a_0 = 0$. Let us  also  mention  the case
where the nonlinearity $H(x,s,\xi)$ has the ``good sign property'',
namely satisfies
\begin{equation}\label{0H1}
- H(x,s,\xi)\,{\rm{sign}}\,(s)\geq 0.
\end{equation}
In this case,  when $a_0\leq 0$ and when $f$  belongs to
$H^{-1}(\Omega)$, L.~Boccardo, F.~Murat and\ 
J.-P.~Puel in~\cite{bmp1} and A.~Bensoussan, L.~Boccardo and F.~Murat in~\cite{bens} proved the existence of at least one solution of
(\ref{02.1}) which belongs to $H_0^1(\Omega)$. 
\vspace{0.8cm}
\par In contrast with the cases (\ref{01.5}) and  (\ref{01.31}),  the present paper is concerned with the case (\ref{01.3}) where $a_0\geq 0$ 
and $a_0 \neq 0$.
\par In this setting we are only aware of four papers, which are recent.  In   \cite{jeanjean},   L.~Jeanjean and B.~Sirakov  proved the existence
of at least two solutions of (\ref{02.1})  (which moreover  belong
to $L^\infty(\Omega)$)  when  $A(x)=Id$,
 $ H(x,s,\xi)=\mu |\xi|^2$, $\mu>0$, $f\in L^q(\Omega)$, $q>\frac{N}{2}$,   $f\geq 0$,
 and    $a_0\in L^q(\Omega)$,  $a_0\geq 0$,  $a_0 \neq 0$, with  $\|f\|_{ L^q(\Omega)}$ and $\|a_0\|_{ L^q(\Omega)}$
 sufficiently small. In \cite{arcoya},
 D.~Arcoya, C.~De Coster, L.~Jeanjean and K.~Tanaka  proved the existence of a continuum $(u,\lambda)$ of solutions  (with $u$ which moreover  belongs to $L^\infty(\Omega)$)
 when $A(x)=Id$,  $H(x,s,\xi)=\mu (x)\,|\xi|^2$, with   $\mu \in L^\infty(\Omega)$ ,    $\mu (x)\geq \mu>0$,
    $f\in L^q(\Omega)$, $q>\frac{N}{2}$,
  $f\geq 0$,  $f\neq 0$  and $a_{0}(x)=\lambda a_0^\star (x)$  with  $a_0^\star \in L^q(\Omega)$, $a_0^\star\geq 0$
  and $a_0^\star\neq0$; moreover, under some  further conditions on $f$,
   these authors proved that this continuum is defined for $\lambda\in ]-\infty\,,\,\lambda_0]$
   with $\lambda_0>0$, 
    and that there are at least two nonnegative solutions of (\ref{02.1}) when $\lambda >0$ is sufficiently small.
In \cite{souplet},  
in a similar setting, assuming only that $\mu (x) \geq 0$ but that the supports of $\mu$ and of $a_0^\star$ have a nonempty intersection and that $N\leq 5$, 
P.~Souplet proved the existence of a continuum $(u,\lambda)$ of solutions, 
and that there are at least two nonnegative solutions of (\ref{02.1}) when $\lambda > 0$ is sufficiently small.   
 In \cite{jeanjean ramos}, L. Jeanjean
 and H. Ramos Quoirin proved the existence of two positive solutions
 (which moreover belong to $L^\infty (\Omega)$)  when $A(x)=Id$, $H(x,s,\xi)=\mu
 |\xi|^2$, $\mu >0$, $f\in L^q(\Omega)$, $q>\frac{N}{2}$, $f\geq 0$, $f\neq
 0$, and $a_0\in C(\overline{\Omega})$ which  can change sign with $a_0^+\neq
 0$ and which satisfies the so called ``thick zero set condition'',
  when  the first eigenvalue of the operator $-\Delta -(a_0+\mu f)$ in $H_0^1(\Omega)$
 is positive.
\vspace{0.8cm}
\par With respect to the results obtained in the four latest papers, we prove in the present paper,
 as said above, the existence of (only) one solution of (\ref{02.1}) in the case (\ref{01.3}) ($a_0\geq 0$) when $a_0$ and $f$ satisfy
the two smallness conditions  (\ref{A1}) and  (\ref{A3}),  but our
result is obtained in the general case of a nonlinearity
$H(x,s,\xi)$ which satisfies only (\ref{0H}),  with $f\in
L^{N/2}(\Omega)$ and  with $a_0\in L^q(\Omega)$, $q> \frac{N}{2}$.
Moreover, the method which allows us to prove this result continues
{\it mutatis mutandis}   to work in the nonlinear monotone case
where the linear operator $-{\rm div}\,(A(x)Du)$ is replaced by a
Leray-Lions operator $-{\rm div}\,(a(x,u, Du))$  working in
$W_0^{1,p}(\Omega)$, for some $1<p<N$,  and where the quasilinear
term
 $H(x,u,Du)$ has $p$-growth in $|Du|$.
\vspace{0.8cm}
\par Let us now describe the contents of the present paper.
\par  The precise statement of our result  is given in Section 2   (Theorem \ref{theo}), as well as
the precise assumptions under which we are able to prove it. These
conditions in particular include the two smallness conditions (\ref{A1}) and (\ref{A3}).
\par Our method for proving Theorem \ref{theo} is based on an
equivalence result (see Theorem \ref{theoequiv})
   that we state in Section 3 once we have introduced the functions $K_\delta (x,s,\zeta)$  and $g_\delta(s)$ (see  (\ref{def k}) and
  (\ref{def g})) and made some technical remarks on them. This result is  very close to the equivalence result given in the paper \cite{fm1}  by V. Ferone and F. Murat.
 \par This equivalence result implies that in order to prove the existence of
 a solution $u$ of (\ref{02.1}) which  satisfies (\ref{01.4}) and
 (\ref{01.55}),  it is equivalent to prove  (see Theorem \ref{theo34}) the existence
of a function $w$ defined by (\ref{3.333}), i.e.
\begin{equation}\label{3.29}
w=\dfrac{1}{\delta_0}(e^{\delta_0 |u|}-1)\, {\rm{sign}}(u),
\end{equation}
which satisfies (\ref{CCbis}), i.e.
\begin{equation}\label{0CCbis}
\left\{\begin{array}{l} w\in H_{0}^{1}(\Omega),
\\
\\
-{\rm{div}}(A(x)Dw)+K_{\delta_{0}}(x,w,Dw)\,{\rm{sign}}(w)\,=
\\
\\
=(1+\delta_{0}|w|)\,f(x)+a_{0}(x)\, w+a_{0}(x)\, g_{\delta_{0}}(w)\,{\rm{sign}}(w)
\;\;\; \textrm{in}\;\;\;\;\mathcal{D}'(\Omega),
\end{array}\right.
\end{equation}
and the estimate  (\ref{CCter}), i.e.
\begin{equation}\label{1.13 bis}
\|w\|_{H_{0}^{1}(\Omega)}\leq Z_{\delta_0},
\end{equation}
(which is nothing but (\ref{01.55})).
\par Our goal thus becomes   to prove  Theorem  \ref{theo34}, namely to prove the existence of a solution $w$ which satisfies
 (\ref{0CCbis}) and (\ref{1.13 bis}).
 \par Problem (\ref{0CCbis}) is very similar to problem (\ref{02.1}), since it involves a term $-K_{\delta_0}(x, w, Dw)\,{\rm sign}(w)$  which has quadratic growth in  $Dw$, as well as a zeroth order term
$\delta_0\,|w| f(x) +\break
+ \, a_{0}(x)\, w+a_{0}(x)\, g_{\delta_0}(w)\, {\rm{sign}}(w) $. But this problem is
also very different  from   (\ref{02.1}), since the term
$-K_{\delta_0}(x, w, Dw)\,{\rm sign} (w)$ with quadratic growth has
now the ``good sign property'' (see (\ref{0H1})), since
$K_{\delta_0}(x, s, \xi)$ satisfies
$$K_{\delta_0}(x, s, \xi) \geq 0,$$
while the zeroth order term is now   no more a linear but   a
semilinear term   with $|s|^{1+\theta}$  growth (see (\ref{6.16}))
due to presence of the term  $a_{0}(x)\, g_{\delta_0}(w)\,{\rm{sign }}(w)$.
\par  We will prove Theorem  \ref{theo34} essentially by applying
Schauder's fixed point theorem. But there  are    some difficulties
to do it directly, since the term with quadratic growth
$K_{\delta_0}(x, w, Dw)\,{\rm sign}(w)$   only belongs to
$L^1(\Omega)$ in general.   We therefore begin by defining  an
approximate problem  (see  (\ref{care}))  where $K_{\delta}(x, w,
Dw)$ is remplaced by its truncation at height $k$, namely
$T_k(K_{\delta}(x, w, Dw))$, and we prove (see Theorem \ref{care})
that   if  $f$   and $a_0$  satisfy the two smallness conditions
(\ref{A1}) and  (\ref{A3}), this approximate problem has at least
one solution $w_k$ which satisfies the a priori estimate

\begin{equation}\label{01.15}
\|w_k\|_{H_{0}^{1}(\Omega)}\leq Z_{\delta_0}.
\end{equation}
 This result,  which  is proved in Section 4, is    obtained by applying Schauder's fixed point theorem in a classical way.
\par We then pass to the limit as $k$ tends to infinity and we prove in Section 5  that (for a subsequence of $k$) $w_k$
tend to some $w^\star$ strongly in $H_0^1(\Omega)$ 
(see (\ref{5.1bis}))  
and that this $w^\star$   is a solution of
(\ref{0CCbis}) which also satisfies (\ref{1.13 bis}) (see End of the
proof of Theorem \ref{theo34}).
\par This completes the proof of  Theorem \ref{theo34}, and therefore proves Theorem \ref{theo}, as announced.
\vspace{0.8cm}
\par This proof follows along the lines of the proof used  by V.~Ferone and F.~Murat in \cite{fm}  in the case where  $a_0=0$.  As mentionned  above, this method can be applied {\it mutatis mutandis} to the nonlinear case where the linear operator $-  {\rm div}(A(x)Du)$ is replaced by a Leray-Lions operator  $-  {\rm div}(a(x,u,Du))$ working in $W_0^1(\Omega)$ for some $1<p<N$, and where the quasilinear term $H(x,u,Du)$ has $ p$-growth in $|Du|$,  as it was done   in \cite{fm1}    by V.~Ferone and F.~Murat
in this nonlinear setting when  $a_0=0$.    This   will be the goal of
our next paper \cite{hamour}.

\section{Main result}
In this paper we consider the following quasilinear problem
\begin{equation}\label{2.1}
\left\{\begin{array}{ll} u\in H_{0}^{1}(\Omega),&
\\ &
\\
-\textrm{div}\,(A(x)Du)=H(x,u,Du)+a_{0}(x)\, u+f(x)& {\rm in}\;\;\;
\mathcal{D}'(\Omega),
\end{array}
\right.
\end{equation}
where the set $\Omega$ satisfies (note that no regularity is assumed
on the boundary of $\Omega$)
\begin{equation}\label{2.2}
\Omega \;\;\; \textrm{ is a bounded open subset of }
\mathbb{R}^{N},\; N \geq 3,\qquad\qquad\qquad\qquad\qquad
\end{equation}
 ${\textrm{where the matrix \,}}A {\textrm { is a coercive matrix with bounded measurable coefficients, i.e. }}$
\begin{equation}
\left\{\begin{array}{l}\label{2.3}
 A\in(L^{\infty}(\Omega))^{N\times N},
\\
\\
\exists \alpha >0,\;\;\; A(x)\,\xi\xi\geq \alpha |\xi|^{2}\;\;\;
\textrm{a.e.}\; x\in\Omega,\;\;\;\forall \xi \in \mathbb{R}^{N},
\end{array}
\right.
\end{equation}
where the function $H(x,s,\xi)$ is a Carath\'{e}odory function with
quadratic growth in $\xi$, and more precisely satisfies
\begin{equation}\label{2.6}
\left\{\begin{array}{ll} H:\Omega \times
\mathbb{R}\times\mathbb{R^{N}}\rightarrow \mathbb{R}
\textrm{ is a Carath\'{e}odory function such that}&\\
&\\
-c_{0}\, A(x)\, \xi\xi\leq H(x,s,\xi)\,\textrm{sign(\textit{s})}\leq
\gamma\,A(x)\,\xi\xi , \;\;\;\textrm{a.e.}\; x\in\Omega,\;\;\;
 \forall s\in\mathbb{R},\;\;\;\forall\xi \in\mathbb{R}^{N},&\\
&\\
\textrm{where } \gamma> 0 \textrm{ and }c_{0}\geq0,&\\
\end{array}
\right.
\end{equation}
where \;$\rm{sign}:\mathbb{R}\rightarrow\mathbb{R}$ denotes the
function defined by
\begin{equation}\label{def sign}
{\textrm{sign}}(s)=\left\{
\begin{array}{cc}
+1&\;\;\; {\rm if }\;\;\;s>0, \\
0&\;\;\; {\rm if }\;\;\;s=0, \\
-1&\;\;\; {\rm if }\;\;\;s<0,
\end{array}
\right.
\end{equation}
where the coefficient $a_0$ satisfies
\begin{equation}\label{2.5}
a_{0}\in L^{q}(\Omega) \;\;\;{\rm for
\;some}\;\;\;q>\frac{N}{2},\;\;\; a_{0}\geq 0,\;\;\; a_{0}\neq0,
\end{equation}
as well as  the technical assumption 
(note that, since
$\dfrac{N}{2}<\dfrac{2N}{6-N}$ when $3\leq N\leq 6$
and since $\Omega$ is bounded,
this assumption can be made without loss of
generality once hypothesis (\ref{2.5}) is assumed)
\begin{equation}\label{2.5bis}
\dfrac{N}{2}<q<\dfrac{2N}{6-N}\quad {\rm{when}}\quad 3\leq N\leq 6,
\end{equation}
and finally where
\begin{equation}\label{2.4}
f\in L^{N/2}(\Omega),\quad f\neq0.
\end{equation}
%
\bigskip
\par Since $N \geq 3$, let $2^{*}$ be the Sobolev's exponent defined by
$$\frac{1}{2^{*}}=\frac{1}{2}-\frac{1}{N},$$
 and let $C_{N}$ be the Sobolev's constant defined as the best constant such that
\begin{equation}\label{2.7'}
\|\varphi\|_{2^{*}} \leq C_{N}\|D\varphi\|_{2},\;\;\;
\forall\varphi\in H_{0}^{1}(\Omega).
\end{equation}
\par We claim that in view of (\ref{2.5}) and (\ref{2.5bis}), one has
\begin{equation}\label{2.10}
 0<\frac{2^{*}}{q'}-2<1,
\end{equation}
where $q'$ the H\"{o}lder's conjugate of the exponent $q$, i.e.
 $$\frac{1}{q'}+\frac{1}{q}=1;$$
indeed easy computations show that
$$
0< \dfrac{2^{*}}{q'}-2\;\;\;\Longleftrightarrow
\;\;\;q>\dfrac{N}{2},
$$
$$
\dfrac{2^{*}}{q'}-2<1 \;\;\;\Longleftrightarrow
\;\;\;\dfrac{1}{q}>\dfrac{6-N}{2N},
$$
where the latest inequality is satisfied when $N > 6$ and is equivalent
to $q<\dfrac{2N}{6-N}$ when $N\leq 6$ (see (\ref{2.5bis})).
\\
 \indent We now define the number
$\theta$ by
\begin{equation}\label{b}
 \theta=\frac{2^{*}}{q'}-2.
\end{equation}
In view of (\ref{2.10}) we have
\begin{equation}\label{c}
 0<\theta <1.
\end{equation}
\bigskip
 \par  Since $\Omega$ is bounded, we equip the space
$H_{0}^{1}(\Omega)$ with the norm
\begin{equation}\label{2.7"}
\|u\|_{H_{0}^{1}(\Omega)}=\|Du\|_{L^{2}(\Omega)^N}.
\end{equation}
\bigskip
\par We finally assume that $f$ and   $a_0$   are sufficiently
small (see Remark \ref{rmq2}), and more precisely that
 \begin{equation}\label{A1}  
 \alpha-C_{N}^{2}\|a_0\|_{N/2}-\gamma C_{N}^{2}\|f\|_{N/2}>0,
\end{equation}
\begin{equation}\label{A3} 
 \|f\|_{H^{-1}(\Omega)}\leq\dfrac{\theta}{1+\theta}\dfrac{(\alpha-C_{N}^{2}\|a_0\|_{N/2}-\gamma C_{N}^{2}\|f\|_{N/2})^{(1+\theta) /\theta}}
 {((1+\theta)GC_{N}^{2+\theta}\|a_0\|_{q})^{1/\theta}},
\end{equation}
where the constant $G$ is defined by (\ref{G bis}).
%
\par Observe
that in place of (\ref{A1})  we could as well have assumed that
\begin{equation*}\label{2.14bis}
 \alpha-C_{N}^{2}\|a_0\|_{N/2}-\gamma C_{N}^{2}\|f\|_{N/2}\geq 0,
\end{equation*}
but that when equality takes places in the latest
inequality, inequality  (\ref{A3}) implies that $f=0$,  and
  then $u=0$ is a solution of (\ref{2.1}), so that the result of Theorem \ref{theo} becomes
  trivial.
\bigskip
\par Our main result is the following Theorem.
\begin{theorem}\label{theo}
Assume that \rm{(\ref{2.2}), (\ref{2.3}), (\ref{2.6}), (\ref{2.5}),
(\ref{2.5bis}) and (\ref{2.4})} \it{hold true}. Assume moreover that the two smallness conditions
{\rm{(\ref{A1})}} and {\rm{(\ref{A3})}} hold true.
\\
\indent Then there exist a constant $\delta_0$  with
$\delta_0\geq\gamma$, and a constant $Z_{\delta_{0}}$, which are
defined in 
Lemma~{\ref{lemma 2}} (see {\rm (\ref{3777bis})},
{\rm (\ref{Phi delta0 bis})})  and {\rm (\ref{6.15bis})}), such that
there exists at least one solution $u$ of {\rm(\ref{2.1})} which
further satisfies
\begin{equation}\label{2.23'}
 (e^{\delta_0|u|}-1)\in H^{1}_{0}(\Omega),
\end{equation}
with
\begin{equation}\label{2.24}
 \|e^{\delta_0|u|}Du\|_{L^2(\Omega)^N}=\left\|\frac{e^{\delta_0|u|}-1}{\delta_0}\right\|_{H^{1}_{0}(\Omega)}\leq Z_{\delta_{0}}.
 \end{equation}
\end{theorem}
\par Our proof of Theorem \ref{theo} is based on an equivalence result (Theorem \ref{theoequiv}) which will be stated and proved in Section
3. This equivalence Theorem will allows us to replace proving  Theorem
\ref{theo} by proving Theorem \ref{theo34} which is
 equivalent to Theorem \ref{theo}.
\begin{remark}\label{rmq2}
{\rm In this Remark, we consider that the open set $\Omega$, the
matrix $A$ and the function $H$ are fixed (and therefore in
particular that the constants $\alpha>0$ and $\gamma>0$ are fixed),
and we consider the functions $a_0$ and $f$ as parameters.
\\
\par Our first set of assumptions on these parameters  (assumptions (\ref{2.5}) and (\ref{2.5bis})) is that $a_0$ belongs
 to $L^{q}(\Omega)$ with $q>\frac{N}{2}$ (and that $q<\frac{2N}{N-6}$ when $3 \leq N \leq 6$; as said above this assumption can be made without loss of generality). 
 This first set of assumptions is essential to ensure (see (\ref{2.10}))
  that the exponent $\theta$ defined by (\ref{b}) satisfies $0<\theta<1$ (see
  (\ref{c})). We also assume $a_0\geq 0$ and $a_0\neq 0$.
\\
\par Our second set of assumptions on these parameters is made of the two smallness conditions (\ref{A1}) and (\ref{A3}).
\par Indeed, if, for example, $a_0$ is sufficiently small such that it satisfies
  $$\alpha-C_{N}^{2}\|a_0\|_{N/2}>0,$$
 then the two smallness conditions (\ref{A1}) and (\ref{A3}) are satisfied if $\|f\|_{N/2}$ (and  therefore $\|f\|_{H^{-1}(\Omega)}$, since
$L^{N/2}(\Omega)\subset H^{-1}(\Omega)$) is sufficiently small.
\\
\indent Similarly, if, for example, $f$ is sufficiently small such that it
satisfies
$$\alpha-\gamma C_{N}^{2}\|f\|_{N/2}>0,$$
then the two smallness conditions (\ref{A1}) and (\ref{A3}) are satisfied if $\|a_0\|_{q}$ is sufficiently small
(which implies,
since  $L^{q}(\Omega)\subset L^{N/2} (\Omega)$,
that $\|a_0\|_{N/2}$ is sufficiently small), since $\|a_0\|_{q}$
appears in the denominator of the right-hand side of (\ref{A3}).
\qquad$\Box$ }
\end{remark}

\begin{remark}\label{rmq2.3}
{\rm
 The definitions of the two constants $\delta_0$ and
$Z_{\delta_0}$ which appear in Theorem  \ref{theo} are given in (the
technical) Appendix 6 (see Lemma \ref{lemma 2}).
These definitions
are based on the properties of the family of functions $\Phi_\delta$ (see
(\ref{Phidelta bis})) which look like convex parabolas (see Figure 2 and Remark~\ref{rmq6.2bis}):
the constant $\delta_0$ is the unique value of the parameter $\delta$ for which
the function $\Phi_{\delta_0}$ has a double zero, and
$Z_{\delta_0}$ is the value of this double zero. The two smallness
conditions (\ref{A1}) and (\ref{A3}) ensure that $\delta_0$
satisfies $ \delta_0 \geq \gamma$, a condition which is essential in
our proof. 
\\
\indent In Remark \ref{rmq3.9} we try to explain where the two
smallness conditions (\ref{A1}) and (\ref{A3}) come from.
\\
\indent In Remark \ref{rmq 3.10}, we explain why we have chosen to
state Theorem \ref{theo34} with $\delta=\delta_0$ rather than with a
fixed $\delta$ with $\gamma\leq\delta\leq\delta_0$.
\qquad $\Box$ }
\end{remark}

\begin{remark}\label{rmq2.4}
{\rm
In assumption (\ref{2.2}) we have assumed that $N \geq 3$, 
because 
we use the Sobolev's embedding (\ref{2.7'}).
All the proofs of the present paper can nevertheless be easily adapted to the cases where $ N = 1$ and $N =2$, 
providing similar results,
by using the fact that $H^1_0 (\Omega) \subset L^\infty (\Omega)$ when $N = 1$ 
and that $H^1_0 (\Omega) \subset L^p (\Omega)$ for every $p < + \infty$  when $N = 2$, 
and by replacing assumption $q > { \dfrac{N}{2}}$ made in (\ref{2.5}) when $N \geq 3$ by $q = 1$ when $ N = 1$ and by  $q > 1$ when $N =2$,
and the assumption made in (\ref{2.4}) that $f \in L^{N/2} (\Omega)$ by $f \in L^1 (\Omega)$ when $ N = 1$ and by  $f \in L^m (\Omega)$ with $m > 1$ when $N =2$ (and also replacing  the norm $\| f \|_{N/2}$ by the corresponding norm).

\par In assumption (\ref{2.5}) and (\ref{2.4}) we have assumed that $a_{0}\neq0$ and that $f\neq0$.
Indeed the case where $a_{0} = 0$ has been treated by V.~Ferone and F.~Murat in \cite{fm},
and in the case where $f = 0$, then $u = 0 $ is a solution of (\ref{2.1}) and the results of the present paper become trivial.
}
\qquad $\Box$
\end{remark}

\section{An equivalence result}
 The main results of this Section are Theorems \ref{theoequiv} and
 \ref{theo34}. 
 In contrast, Remarks \ref{rmq3.1}, \ref{Rmq3.2} and \ref{rmq 34}
 and Lemma \ref{lemma3 2} can be considered as technical results.
 \par Indeed, as  said above, the proof of Theorem \ref{theo} is based on
 the equivalence result of Theorem \ref{theoequiv}  that we state and prove
in this Section. This equivalence Theorem in particular implies that
Theorem \ref{theo34},  which we state at the end of this Section, is
equivalent to Theorem \ref{theo}. Theorem \ref{theo34} will be
proved in Sections 4 and 5. 
 \par This Section also includes Remark
\ref{rmq3.9} in which  we try to explain where the two smallness
conditions (\ref{A1}) and (\ref{A3})
 come from, as well as Remark \ref{rmq 3.10} where we explain why we have
 chosen to state  Theorem \ref{theo34} for $\delta=\delta_0$.
 \\
 \par In this Section (as well as in the whole of the present paper) we always assume that
\begin{equation}\label{3.0}
\delta>0.
\end{equation}
\par Let us first proceed with a formal computation. 
\par If $u$ is a solution of
\begin{equation}\label{th1}
\left\{\begin{array}{ll}
-\textrm{div}\,(A(x)Du)=H(x,u,Du)+f(x)+a_{0}(x)\, u & {\rm in }\;\;\; \Omega,\\
&
\\u =0 \;\;\;{\rm on}\;\;\;\partial\Omega, &
\end{array}
\right.
\end{equation}
and if we formally define the function $w_{\delta}$ by
\begin{equation}\label{w}
w_{\delta}=\frac{1}{\delta}\left(e^{\delta|u|}-1\right)\textrm{sign}(u),
\end{equation}
where the function sign is defined by (\ref{def sign}), we have, at
least formally,
\begin{equation}\label{4.3}
\left\{\begin{array}{l} e^{\delta|u|}=1+\delta|w_\delta|,\qquad
|u|=\dfrac{1}{\delta}\,\log(1+\delta|w_{\delta}|),
\qquad\textrm{sign}(u)=\textrm{sign}(w_{\delta}),
\\
\\
Dw_{\delta}=e^{\delta|u|}Du,\quad A(x) Dw_{\delta}=e^{\delta|u|}A(x) Du,
\\
\\
-\textrm{div}\,(A(x)Dw_{\delta})=-\delta e^{\delta
 |u|}\,A(x)DuDu\,\textrm{sign}(u) -e^{\delta |u|}\big
(\textrm{div}\,(A(x)Du)\big),
\end{array}\right.
\end{equation}
and therefore $w_{\delta}$ is, at least formally, a solution of
\begin{equation}\label{equa2'}
\left\{\begin{array}{l}
 -\textrm{div}\,(A(x)Dw_{\delta})=
\\
\\
=-\delta e^{\delta |u|}A(x)DuDu\,{\rm sign}(u)+e^{\delta
|u|}H(x,u,Du) +e^{\delta |u|}\,f(x)+e^{\delta |u|}\,  a_{0}(x)\, u=
\\
\\

= -K_\delta(x,w_{\delta},Dw_{\delta})\,\textrm{sign}(w_{\delta})\,+
\\
\\
+(1+\delta|w_{\delta}|)\,f(x) +\,a_{0}(x)\, w_{\delta}+a_{0}(x)\, g_{\delta}\,(w_{\delta})\,\textrm
{sign}(w_{\delta}) \;\;\;\rm{in}\;\;\;\Omega ,\\
\\
w_{\delta}=0 \;\;\; \rm{on}\;\;\;\partial\Omega,
\end{array}\right.
\end{equation}
whenever the functions  $K_{\delta}:
\Omega\times\mathbb{R}\times\mathbb{R}^{N}\rightarrow \mathbb{R}$
 and $g_{\delta}: \mathbb{R}\rightarrow
\mathbb{R}$ are defined by the formulas
\begin{equation}\label{def k}
\left\{\begin{array}{l}
K_{\delta}(x,t,\zeta)=\\
\\
=\dfrac{\delta}{1+\delta|t|}\,A(x)
\zeta\zeta-(1+\delta|t|)H(x,\dfrac{1}
{\delta}\,\textrm{log}(1+\delta|t|)\,\textrm{sign}(\textit{t}),\dfrac{\zeta}{1+\delta|t|})
{\,\rm{sign}}(t),
\\
\\
{\rm a.e.}\; x\in\Omega, \;\;\;\forall
t\in\mathbb{R},\;\;\;\forall\zeta\in\mathbb{R}^{N},
\end{array}\right.
\end{equation}
and
\begin{equation}\label{def g}
g_{\delta}(t)=-|t|+\frac{1}{\delta}(1+\delta |t|)\log(1+\delta
|t|),\;\;\;\forall t\in\mathbb{R},
\end{equation}
\noindent which is equivalent to
\begin{equation}\label{3.6bis}
t+g_{\delta}(t)\,{\rm{sign}}(t)=(1+\delta |t|)
 \dfrac{1}{\delta}\log(1+\delta
|t|)\,{\rm{sign}}(t),\quad \forall t\in\mathbb{R}.
\end{equation}
\par Conversely, if $w_{\delta}$ is a solution of (\ref{equa2'}), and if we formally define the function $u$ by
\begin{equation}\label{def u'}
u=\frac{1}{\delta}\log(1+\delta|w_{\delta}|)\,\textrm{sign}(w_{\delta}),
\end{equation}
the same formal computation easily shows that $u$ is a solution of (\ref{th1}).\\
\par The main goal of this Section is to transform this formal equivalence into a mathematical result,
namely Theorem \ref{theoequiv}. We begin with  three  remarks  on
the functions $K_\delta$ and $g_\delta$, namely Remarks
\ref{rmq3.1}, \ref{rmq3.2} and   \ref{rmq3.3}.

\vfill
\eject

\begin{remark}\label{rmq3.1}
\rm{ Observe that, because of the inequality (\ref{2.6}) on the function
$H$, and because of the coercivity (\ref{2.3}) of the matrix $A$,
one has
\begin{equation}\label{def K1}
\left\{\begin{array}{l} (c_0+\delta)A(x)\zeta\zeta\geq \dfrac{c_0+\delta}{(1+\delta|t|)}A(x)\zeta\zeta =
\\
\\
=\dfrac{\delta}{(1+\delta|t|)}A(x)\zeta\zeta+
(1+\delta|t|)\dfrac{c_0}{(1+\delta|t|)^2}A(x)\zeta\zeta\geq
\\
\\
\geq K_{\delta}(x,t,\zeta)\,\geq
\\
\\
\geq\dfrac{\delta}{(1+\delta|t|)} A(x)\zeta\zeta
 -(1+\delta|t|)\dfrac{\gamma
}{(1+\delta|t|)^{2}}A(x)\zeta\zeta=
\\
\\
= \dfrac{(\delta-\gamma)}{(1+\delta|t|)}A(x)\zeta\zeta\geq -|\delta -\gamma|\,A(x) \zeta\zeta,
\\
\\
{\rm a.e.}\; x\in\Omega, \;\;\;\forall
t\in\mathbb{R},\;\;\;\forall\zeta\in\mathbb{R}^{N},
\;\;\;\forall{\delta>0}.
\end{array}\right.
\end{equation}
When $\delta\geq\gamma$, this computation in particular implies that
\begin{equation}\label{def K2}
\begin{array}{l}
(c_0+\delta)A(x)\zeta\zeta \geq K_{\delta}(x,t,\zeta)\geq 0 \quad
{\rm a.e.}\; x\in \Omega,\;\;\;\forall t\in
\mathbb{R},\;\;\;\forall \zeta\in\mathbb{R}^{N}\;\;\; {\rm if}
\;\;\;\delta\geq\gamma.
\end{array}
\end{equation}
}
\hspace{13.8cm} \qquad $\Box$
\end{remark}
\begin{remark}\label{Rmq3.2}\label{rmq3.2}
\rm{In this technical Remark we prove that the functions
$K_{\delta}(x,w,Dw)$ and\break
 $K_{\delta}(x,w,Dw)\,{\rm{sign}}(w)$
 are correctly defined 
 and are measurable functions when  $w\in H^{1}(\Omega)$, and we prove their continuity with respect to  the almost
 everywhere convergence of $w$ and $Dw$ (see Lemma \ref{lemma3 2}).
\par
Note that the function $K_{\delta}(x,t,\zeta)$ defined by (\ref{def
k}) and the function
 $K_{\delta}(x,t,\zeta)\,{\rm{sign}}(t)$ are not
 Carath\'{e}odory functions, because their definitions involve the function ${\rm{sign}}(t)$,
  which it is not a Carath\'{e}odory function since is not continuous at $t=0$.
 This lack of continuity in $t=0 $ is however the only obstruction for the functions
  $K_{\delta}(x,t,\zeta)$ and $K_{\delta}(x,t,\zeta)\,{\rm sign}(t)$ to be Carath\'{e}odory functions, and
 \begin{equation}\label{3.10}
\left\{\begin{array}{l}
 {\rm for\; every} \; w\in H^{1}(\Omega),
\\\\
 {\rm  the\; functions}\;   K_{\delta}(x,w,Dw)\; {\rm  and}\; K_{\delta}(x,w,Dw)\,{\rm sign}(w) \; {\rm are\; well\; defined}\;
\\\\  {\rm and \;are\; measurable \;functions},
 \end{array}
\right.
\end{equation}
as it immediately results  from the two formulas
 \begin{equation}\label{h'bis}
\left\{
\begin{array}{l}
K_{\delta}(x,t,\zeta)=\\
\\={\dfrac{\delta}{1+\delta
|t|}\,A(x)\zeta\zeta -({1+\delta
|t|})\,H(x,\dfrac{1}{\delta}\log(1+\delta|t|)\,{\rm
sign}(\textit{t}) ,\dfrac{\zeta}{1+\delta|\textit{t}|})\,\rm
sign}(t)\,=
\\
\\
= \dfrac{\delta}{1+\delta |t|}\,A(x)\zeta\zeta\, +
\\
\\
+\;\chi_{\{t<0\}}(t)\,(1+\delta |t|)\,
H(x,-\dfrac{1}{\delta}\log(1+\delta|t|),\dfrac{\zeta}{1+\delta|\textit{t}|})-\chi_{\{t=0\}}(t)\,0
\,+
\\
\\
 -\;\chi_{\{t>0\}}(t)\,(1+\delta
|t|)\,H(x,\dfrac{1}{\delta}\log(1+\delta|t|),\dfrac{\zeta}{1+\delta|\textit{t}|}),
\\
\\
{\rm a.e.}\; x\in\Omega, \;\;\;\forall
t\in\mathbb{R},\;\;\;\forall\zeta\in\mathbb{R}^{N},
\end{array}
\right.
\end{equation}


 \begin{equation}\label{h'}
\left\{
\begin{array}{l}
K_{\delta}(x,t,\zeta)\,{\rm sign}(\textit{t})=\\
\\
={\rm sign}(t)\,\dfrac{\delta}{1+\delta |t|}\,A(x)\zeta\zeta
-({1+\delta |t|})\,H(x,\dfrac{1}{\delta}\log(1+\delta|t|)\,{\rm
sign}(\textit{t}),\dfrac{\zeta}{1+\delta|\textit{t}|})\,=
\\
\\

=-\;\chi_{\{t<0\}}(t)\,\dfrac{\delta}{1+\delta |t|}\,A(x)\zeta\zeta
+\chi_{\{t=0\}}(t)\,0+ \chi_{\{t>0\}}(t)\,\dfrac{\delta}{1+\delta
|t|}\,A(x)\zeta\zeta \,+
\\
\\
-\;\chi_{\{t<0\}}(t)\,(1+\delta |t|)\,H(x,-\dfrac{1}{\delta}
\log(1+\delta|t|),\dfrac{\zeta}{1+\delta|\textit{t}|})-\chi_{\{t=0\}}(t) \, H(x, 0, \zeta) \,+
\\
\\
-\;\chi_{\{t>0\}}(t)\,(1+\delta
|t|)\,H(x,\dfrac{1}{\delta}\log(1+\delta|t|),\dfrac{\zeta}{1+\delta|\textit{t}|}),
\\
\\
{\rm a.e.}\; x\in\Omega, \;\;\;\forall
t\in\mathbb{R},\;\;\;\forall\zeta\in\mathbb{R}^{N}.
\end{array}
\right.
\end{equation}
\par
Moreover, in view of (\ref{def K1})  one has
\begin{equation}
K_{\delta}(x,w,Dw)\in L^1(\Omega),\;\;\;K_{\delta}(x,w,Dw)\,{\rm
sign}(\textit{w})\in L^1(\Omega),\;\;\;\forall w\in
H^1(\Omega),\;\;\;\forall \delta >0.
\end{equation}

\hspace{13.4cm} \qquad $\Box$
\\ 
\par
One also has  the following convergence result. }
\end{remark}

\begin{lemma}\label{lemma3 2}
Consider a sequence $w_n$ such that
\begin{equation}
\begin{array}{l}
 w_{n}\in H^{1}(\Omega),\;w\in H^{1}(\Omega),\quad
 w_{n}\rightarrow w \quad {\rm a.e. \;\;\;in}\;\;\;\Omega,\; \;\;Dw_{n}
 \rightarrow Dw\;\;\; {\rm a.e. \;\;\;in\;\;\;}\Omega.
 \end{array}
 \end{equation}
\indent Then
\begin{equation}\label{consconv}
\left\{
\begin{array}{l}
 K_{\delta}(x,w_{n},Dw_{n})\rightarrow K_\delta(x,w,Dw)\;\;\;{\rm  a.e.\;\;\;in\;\;\;}\Omega, \\
 \\
 K_{\delta}(x,w_{n},Dw_{n})\,{\rm sign}(w_{n})\rightarrow K_\delta(x,w,Dw)\,{\rm sign}(w)\;\;\;{\rm
 a.e.\;\;\;in\;\;\;}\Omega.
 \end{array}
 \right.
 \end{equation}
\end{lemma}
\proof On the first hand, we have the following almost everywhere
convergences
 \begin{equation}\label{l'}
\left\{
\begin{array}{l}
\dfrac{\delta}{1+\delta
|w_n|}A(x)Dw_nDw_n\rightarrow\dfrac{\delta}{1+\delta
|w|}A(x)DwDw\;\;\; {\rm {a.e.\;\;\; in}}\;\;\;\Omega,
\\
\\
\displaystyle(1+\delta|w_n|)H(x,-\dfrac{1}{\delta}\log(1+\delta|w_n|),\dfrac{Dw_n}{1+\delta
|w_n|})\rightarrow
\\
\rightarrow(1+\delta|w|)H(x,-\dfrac{1}{\delta}\log(1+\delta|w|),\dfrac{Dw}{1
+\delta |w|})\;\;\;\;{\rm a.e.\;\;\; in}\;\;\;\Omega,
\\
\\
H(x,0,Dw_n)\rightarrow H(x,0,Dw)\;\;\;{\rm a.e.\;\;\; in}\;\;\;\Omega,
\\
\\
\displaystyle(1+\delta|w_n|)H(x,\dfrac{1}{\delta}\log(1+\delta|w_n|),\dfrac{Dw_n}{1+\delta
|w_n|})\rightarrow
\\
\rightarrow
(1+\delta|w|)H(x,\dfrac{1}{\delta}\log(1+\delta|w|),\dfrac{Dw}{1+\delta
|w|})\;\;\;\;{\rm a.e.\;\;\; in}\;\;\;\Omega.
\end{array}
\right.
\end{equation}
\par
On the other hand, for almost every $x$ fixed in the set  $\{y\in
\Omega:w(y)>0\}$, the assertion
$w_n(x)\rightarrow w(x)$
 implies, since  $w(x)>0$, that one has  $w_n(x)>0$ for $n$ sufficiently large (depending on $x$), and therefore that, for some $n> n^\star(x)$, one has
\begin{equation}
\left\{
\begin{array}{l}
\chi_{\{w_n<0\}}(w_n(x))=0=\chi_{\{w<0\}}(w(x)) \;\;\;{\rm
for}\;\;\;  n>n^\star (x),
\\
\\
\chi_{\{w_n=0\}}(w_n(x))=0=\chi_{\{w=0\}}(w(x)) \;\;\;{\rm for}
\;\;\; n>n^\star(x),
\\
\\
\chi_{\{w_n>0\}}(w_n(x))=1=\chi_{\{w>0\}}(w(x)) \;\;\;{\rm for}
\;\;\; n>n^\star(x).
\end{array}
\right.
\end{equation}
These convergences and (\ref{l'}) imply that
\begin{equation*}
\left\{
\begin{array}{l}
 K_{\delta}(x,w_{n},Dw_{n})\rightarrow K_\delta(x,w,Dw) \;\;\; {\rm a.e.\;\;\; in\;\;\;} \{y\in \Omega:w(y)>0\}, \\
 \\
 K_{\delta}(x,w_{n},Dw_{n})\,{\rm sign}(w_{n})\rightarrow K_\delta(x,w,Dw)\,{\rm sign}(w) \;
 \;\; {\rm a.e. \;\;\; in\;\;\;} \{y\in \Omega:w(y)>0\} .
 \end{array}
 \right.
 \end{equation*}
\par
The same proof gives the similar result in the set $\{y\in
\Omega:w(y)<0\}$.
\\
\par
 The proof  in  the set $\{y\in \Omega:w(y)=0\}$ is a little bit more delicate.
 Let us first observe that  for $w\in H^1(\Omega)$ one has
\begin{equation}\label{dw=0}
Dw=0\;\;\;{\rm a.e.\;\;\; in}\;\;\;\{y\in \Omega:w(y)=0\},
\end{equation}
and since inequality (\ref{2.6}) on the function $H$ implies that
\begin{equation}\label{hsO}
H(x,s,0)=0\;\;\; {\rm a.e. }\; x \in\Omega,\;\;\;\forall
s\in\mathbb{R},\;s\neq0,
\end{equation}
and therefore by continuity in $s$ that
\begin{equation}\label{h00}
H(x,0,0)=0\;\;\; {\rm a.e.}\; x \in\Omega.
\end{equation}
Then,  on the first hand, theses results and  formulas (\ref{h'bis}) and (\ref{h'}) imply that
\begin{equation}\label{3.18biss}
K_{\delta}(x,w,Dw) = K_\delta(x,w,Dw)\,{\rm{sign}}(w)=0 \;\;\;{\rm
a.e.\;\;\; in}\;\; \; \{y\in \Omega:w(y)=0\}.
\end{equation}
On the other hand, in view of (\ref{dw=0}) and (\ref{h00}), the
four functions which appear in the four limits in (\ref{l'}) vanish
almost everywhere in the set $\{y\in \Omega:w(y)=0\}$.  Even if we
do not know anything about the pointwise limits of
the functions $\chi_{\{w_n<0\}}(w_n(x))$, $\chi_{\{w_n = 0\}}(w_n(x))$ and
$\chi_{\{w_n>0\}}(w_n(x))$ in the set $\{y\in\Omega\,:\,w(y)=0\}$,
 this fact and formulas (\ref{h'bis}) and (\ref{h'}) prove that
\begin{equation}\label{k18}
\left\{
\begin{array}{l}
 K_{\delta}(x,w_{n},Dw_{n})\rightarrow 0\;\;\;{\rm a.e.\;\;\; in \;\;\;}\{y\in \Omega:w(y)=0\}, \\
 \\
 K_{\delta}(x,w_{n},Dw_{n})\,{\rm sign}(w_{n})\rightarrow 0\;\;\;{\rm a.e.\;\;\; in \;\;\;}\{y\in \Omega:w(y)=0\}.
 \end{array}
 \right.
 \end{equation}
 \noindent From (\ref{3.18biss}) and (\ref{k18}) we deduce that
\begin{equation}
 \left\{
\begin{array}{l}
 K_{\delta}(x,w_{n},Dw_{n})\rightarrow K_\delta(x,w,Dw)\;\;\;{\rm a.e.\;\;\; in \;\;\;}\{y\in \Omega:w(y)=0\}, \\
 \\
 K_{\delta}(x,w_{n},Dw_{n})\,{\rm sign}(w_{n})\rightarrow
  K_\delta(x,w,Dw)\,\textrm{sign}(w)\;\;\;{\rm a.e.\;\;\; in\;\; \;}\{y\in \Omega:w(y)=0\}.
 \end{array}
 \right.
 \end{equation}
 \par This completes the proof of (\ref{consconv}).
 \qquad $\Box$

\begin{remark}\label{rmq 34} \label{rmq3.3}
\rm{Observe that the function $g_\delta(s)\,{\rm{sign}}(s)$ is a
Carath\'{e}odory function since
 in view of (\ref{aster}) this function is continuous at $s=0$. This allows one to define  $g_\delta(w)\,{\rm{sign}}(w)$
 as a measurable function for every $w\in H^{1}(\Omega).
 $\qquad$\Box$
}
\end{remark}
\par The main result of this Section is the following equivalence Theorem.


\begin{theorem}\label{theoequiv}
Assume that {\rm{(\ref{2.2}), (\ref{2.3}), (\ref{2.6}), (\ref{2.5}),
(\ref{2.5bis}) and (\ref{2.4})}}
 hold true, and let $\delta >0$ be fixed. Let  the functions $K_{\delta}$ and  $g_{\delta}$ be defined by {\rm (\ref{def k})} and   {\rm (\ref{def g})}.
\\
\indent If $u$ is any solution of {\rm(\ref{2.1})} which satisfies
\begin{equation}\label{2.23}
 (e^{\delta|u|}-1)\in H^{1}_{0}(\Omega),
\end{equation}
\noindent then the function $w_\delta$ defined by {\rm(\ref{w})},
namely by
$$
w_{\delta}=\frac{1}{\delta}\left(e^{\delta|u|}-1\right){\,\rm{sign}}(u),
$$
\noindent satisfies
\begin{equation}\label{CC}
\left\{\begin{array}{l} w_{\delta}\in H_{0}^{1}(\Omega),
\\
\\
-\mbox{\rm
{div}}\,(A(x)Dw_{\delta})+K_\delta(x,w_{\delta},Dw_{\delta})\,{\rm{sign}}(w_{\delta})\,=
\\
\\
= (1+\delta|w_{\delta}|)\,f(x) + a_{0}(x)\, w_{\delta}+a_{0}(x)\, g_{\delta}(w_{\delta}){\,\rm{sign}}(w_{\delta})
\;\;\; {in}\;\;\;\mathcal{D}'(\Omega).
\end{array}\right.
\end{equation}
\indent Conversely, if $w_\delta$ is any solution of
{\rm(\ref{CC})}, then the function $u$ defined by {\rm(\ref{def u'})},
namely by
$$
u=\frac{1}{\delta}\log(1+\delta|w_{\delta}|)\,{\rm
sign}(w_{\delta}),
$$
is a solution of {\rm(\ref{2.1})} which satisfies {\rm(\ref{2.23})}.
\end{theorem}


\begin{remark}
{\rm Every term of the equation in (\ref{CC}) 
has a meaning in the sense of distributions: 
indeed the first term of the left-hand side of the equation in (\ref{CC}) belongs to $H^{- 1} (\Omega)$;
on the other hand, the four other terms of the  equation are
measurable functions (see Remarks \ref{Rmq3.2} and \ref{rmq 34}); 
  the second term of the left-hand side of the equation in
(\ref{CC}) belongs to $L^1(\Omega)$  in view of (\ref{def K1}),
while the three terms of the right-hand side of the equation in (\ref{CC}) can be
proved to belong to $(L^{2^\star} \!(\Omega))'$ (see e.g. the proof of (\ref{3.8}) in Remark
\ref{rmq3.9} and the proof of (\ref{4.9}) in the proof of Lemma
\ref{lem}).
\qquad$\Box$ }
\end{remark}


\begin{remark}
{\rm Observe that the equivalence Theorem \ref{theoequiv} holds true
without assuming the two smallness conditions (\ref{A1}) and
(\ref{A3}); moreover one could even have removed
in (\ref{2.3}) the assumption that the matrix $A$ is coercive, and still obtain the same equivalence result.
\\
\indent Note however that Theorem \ref{theoequiv} is an equivalence
result which does not proves neither the existence of a solution of
(\ref{2.1}) nor the existence of a solution of (\ref{CC}), but which
assumes as an hypothesis either  the existence of a solution of
(\ref{2.1}) which also satisfies (\ref{2.23}),  or the existence of
a solution of (\ref{CC}). 
\qquad$\Box$ }
\end{remark}
\par
\noindent
\begin{proof}[Proof of Theorem \ref{theoequiv}]
Define the function $\hat{f}$ by
\begin{equation}\label{FM}
\hat{f}(x)=f(x)+a_{0}(x)\, u(x),
\end{equation}
\noindent In view of (\ref{w}) and of the definition (\ref{def g})
of $g_{\delta }(s)$, one has (see (\ref{3.6bis}) and (\ref{4.3}))
\begin{equation}\label{3.23bis}
 \left\{
\begin{array}{l}
(1+\delta|w_{\delta}|)\, f(x) +
a_{0}(x)\, w_{\delta}+a_{0}(x)\, g_{\delta}(w_{\delta})\,{\rm
sign}(w_{\delta})=
\\
\\
=(1+\delta|w_{\delta}|)\,\big(f(x)+a_{0}(x)\, \dfrac{1}{\delta}\,\log(1+\delta
|w_\delta|)\,{\rm{sign}}(w_\delta)\big)=
\\
\\
=(1+\delta|w_{\delta}|)\,(f(x)+a_{0}(x)\, u(x))=(1+\delta
|w_\delta|)\hat{f}(x).

 \end{array}
 \right.
 \end{equation}
\indent Then Theorem \ref{theoequiv} becomes an immediate
application of Proposition 1.8 of \cite{fm}, once one observes that
\begin{equation}\label{FM'}
\hat{f}\in L^{N/2}(\Omega );
\end{equation}
such is the case in the setting of Theorem \ref{theoequiv}: indeed
$\hat{f}=f+a_0 u$, where $f$ belongs to $L^{N/2}(\Omega)$ by
assumption (\ref{2.4}), and where $a_0u$ also belongs to
$L^{N/2}(\Omega)$, since $a_0$ is assumed to belong to
$L^{q}(\Omega),\;q>\frac{N}{2}$ (see assumption (\ref{2.5})), while
$u$ belongs
 to $L^{r}(\Omega)$ for every $r<+\infty$, since by (\ref{2.23}) $(e^{\delta |u|}-1)$  is assumed
 to belong to $H^{1}_{0}(\Omega)$, hence in particular to $L^{1}(\Omega) $, which implies
 that $ e^{\delta |u|}$ belongs to $L^{1}(\Omega)$.
 Theorem \ref{theoequiv} is therefore proved.
\qquad$\Box$
\end{proof}
\\
\par
From the equivalence Theorem \ref{theoequiv} one immediately
deduces, setting
\begin{equation}\label{3.333}
w=\dfrac{1}{\delta_{0}}(e^{\delta_{0}|u|}-1)\,\mbox{sign}(u)
\end{equation}
and equivalently
\begin{equation}\label{3.334}
u=\frac{1}{\delta_{0}}\log(1+\delta_{0}|w|)\,{\rm{sign}}(w),
\end{equation}
that Theorem \ref{theo} is equivalent to the following Theorem.
\begin{theorem}\label{theo34}
Assume that \rm{(\ref{2.2}), (\ref{2.3}), (\ref{2.6}), (\ref{2.5}),
(\ref{2.5bis}) and (\ref{2.4})} \it{hold true}. Assume moreover that the two smallness conditions
{\rm{(\ref{A1})}} and {\rm{(\ref{A3})}} hold true.
\\
\indent Then there exist a constant $\delta_0$  with
$\delta_0\geq\gamma$, and a constant $Z_{\delta_{0}}$, which are
defined in  Lemma~{\ref{lemma 2}}  (\it see
{\rm(\ref{3777bis})}, {\rm(\ref{Phi delta0 bis})} and {\rm
(\ref{6.15bis})}),
 such that there exists at least one solution $w$ of
\begin{equation}\label{CCbis}
\left\{\begin{array}{l} w\in H_{0}^{1}(\Omega),
\\
\\
-{\rm{div}}(A(x)Dw)+K_{\delta_{0}}(x,w,Dw)\,{\rm{sign}}(w)\,=
\\
\\
=(1+\delta_{0}|w|)\,f(x)+a_{0}(x)\, w+a_{0}(x)\, g_{\delta_{0}}(w)\,{\rm{sign}}(w)\;\;\;
{in}\;\;\;\;\mathcal{D}'(\Omega),
\end{array}\right.
\end{equation}
which satisfies
\begin{equation}\label{CCter}
 \|w\|_{H_{0}^{1}(\Omega)}\leq Z_{\delta_{0}}.
\end{equation}
\end{theorem}

\par The rest  of this paper will therefore be devoted to  the proof of Theorem \ref{theo34}. This will be done in two steps:
first, in Section 4 we will prove the existence
of a solution satisfying (\ref{CCter}) for a problem which
approximates (\ref{CCbis}), see Theorem 4.1;
second, in Section 5,  we will pass to the
 limit in this approximate problem and prove that for a subsequence the limit satisfies (\ref{CCbis}) and (\ref{CCter}).



\begin{remark}\label{rmq3.9}
{\rm{In this Remark we assume that  (\ref{2.2}), (\ref{2.3}),
(\ref{2.6}), (\ref{2.5}),  (\ref{2.5bis})  and  (\ref{2.4}) hold
true. We also assume that the two smallness conditions (\ref{A1})
and  (\ref{A3}) hold true, and we  try to explain how
 these two  conditions  come from an ``a priori estimate'' that one can obtain on the solutions of (\ref{CC}).
 \\
\indent If $w_\delta $  is any  solution of (\ref{CC}), using
$T_k(w_\delta)\in H_{0}^1(\Omega)\cap L^\infty(\Omega)$ as test
function, where $T_{k}:\mathbb{R}\rightarrow\mathbb{R}$ is the usual
truncation at height $k$ defined by
\begin{equation}\label{tk}
T_{k}(s)= \left\{\begin{array}{ll}
- k &  \quad {\rm if}\; \;\; s \leq - k, \\
\; \; \; s &  \quad {\rm if} \;\;\; -k \leq s \leq + k,\\
+ k &  \quad {\rm if}\; \;\; + k \leq s,
\end{array}
\right.
\end{equation}
one has
\begin{equation}\label{3.40}
\left\{\begin{array}{l}
 \displaystyle\int _\Omega\,A(x)Dw_\delta\,
DT_k(w_\delta)+\int_\Omega\,K_\delta(x,w_\delta,Dw_\delta)|T_{k}(w_\delta)|=
\\
\\
= \displaystyle\int _\Omega\,f(x)\,T_{k}(w_\delta)+ \int
_\Omega\,\delta f(x) |w_\delta|\,T_{k}(w_\delta) +\int
_\Omega\,a_{0}(x)\, w_{\delta}\,T_{k}(w_\delta)+
\\
\\ \displaystyle +\int
_\Omega\,a_{0}(x)\, g_{\delta}(w_{\delta})\,|T_{k}(w_\delta)|.
\end{array}
\right.
\end{equation}
\par From now on we  assume in this Remark that $\delta$  satisfies
\begin{equation}\label{3.41}
\gamma\leq\delta\leq\delta_1,
\end{equation}
where $\delta_1$ is defined by (\ref{Ldelta1 bis}) (note that one has 
$\gamma  <\delta_1$, 
see  (\ref{6.18bis})).
\par Since $\delta\geq\gamma$  by (\ref{3.41}),  we deduce from (\ref{def
K2}) that  $K_\delta(x,s,\zeta)\geq 0$, and therefore that the
second term of the left-hand side of (\ref{3.40}) is nonnegative. 
We then use  in the first term of the left-hand side of (\ref{3.40})
the fact that  the matrix $A$ is coercive (see (\ref{2.3})),
in the first term of the right-hand side of (\ref{3.40})
 the fact that
 $f\in L^{N/2}(\Omega)$ (see (\ref{2.4})), which  implies that
$f\in H^{-1}(\Omega)$,   
in the second  and in the third terms of the right-hand side of
(\ref{3.40})
H\"{o}lder's inequality with
$$
\dfrac{1}{\frac{N}{2}}+\dfrac{1}{2^\star}+\dfrac{1}{2^\star}=1,
$$
and finally  in the fourth term of the right-hand side of (\ref{3.40})
the second statement of (\ref{6.17bis}), namely
\begin{equation}\label{6.15 ter}
0\leq g_{\delta}(t)< G|t|^{1+\theta},\;\;\; \forall t\in
\mathbb{R},\; t\neq0,
\;\;\;\forall\delta,\;0<\delta\leq\delta_{1},
\end{equation}
(note that here we use  $\delta\leq\delta_1$, see (\ref{3.41})) and
H\"{o}lder's inequality with
$$
\dfrac{1}{q}+\dfrac{1+\theta}{2^\star}+\dfrac{1}{2^\star}=1
$$
(which results from the definition (\ref{b}) of $\theta$). 
This allows us to obtain an estimate on $T_{k}(w_\delta)$,
in wich we pass to the limit in $k$ to get
\begin{equation}\label{6.15 quat}
\left\{\begin{array}{l} \alpha
\|Dw_\delta\|_2^2<\|f\|_{H^{-1}(\Omega)}\|Dw_\delta\|_2+\delta\|f\|_{N/2}\|w_{\delta}\|^2_{2^\star}
+ \|a_0\|_{N/2}\|w_{\delta}\|^2_{2^\star}
+G\|a_0\|_{q}\|w_{\delta}\|^{2+\theta}_{2^\star}\!,
\\ \\ 
{\rm{if}}\;w_{\delta}\neq0,
\end{array}
\right.
\end{equation}
(note that in view of (\ref{6.15 ter}), inequality (\ref{6.15 quat}) is actually 
a strict inequality). Using Sobolev's inequality (\ref{2.7'}) and
dividing by $\|Dw_\delta\|_2$ this implies that (even in the case
where $w_\delta=0$)
\begin{equation}\label{3.8}
\begin{array}{l}
\alpha \|Dw_\delta\|_2<\|f\|_{H^{-1}(\Omega)}+\delta
C_N^2\|f\|_{N/2}+C_N^2\|a_0\|_{N/2}+GC_N^{2+\theta}\|a_0\|_{q}\|Dw_\delta\|_{2}^{1+\theta}.
\end{array}
\end{equation}
\par In view of the definition (\ref{Phidelta bis}) of the function
$\Phi_\delta$ (see also Figure 2),  we have proved that if
$w_\delta$ is any solution of (\ref{CC}),  one has
\begin{equation}\label{3.43}
\Phi_\delta(\|Dw_\delta\|_2)>0
\;\;\; {\rm{if}}\;\;\; 
\gamma\leq\delta\leq\delta_1.
\end{equation}
\par But by the  definition of $\delta_0$, one has
$$
\Phi_\delta(X)>0,
\;\;\; \forall X\geq0,
\;\;\; \forall\delta, \; 0<\delta\leq\delta_{1},
$$
and therefore inequality (\ref{3.43}) does not imply  anything on
$\|Dw_\delta\|_2$ when $\delta$ satisfies
$\delta_0<\delta\leq\delta_1$. In contrast, when $\delta<\delta_0$,
the strict  inequality (\ref{3.43}) implies that
\begin{equation}\label{3.44}
{\rm{either}}\quad\|Dw_\delta\|_2<\ Y_{\delta}^-\quad
{\rm{or}}\quad\|Dw_\delta\|_2> Y_{\delta}^+\quad{\rm{if}}\quad
\delta<\delta_0,
\end{equation}
where $Y_{\delta}^- <Y_{\delta}^+$ are the two distinct zeros of the function
$\Phi_\delta$ (see Remarks \ref{rmq6.2bis} and  \ref{rmq6.4} and Figure~2), 
while when
$\delta=\delta_0$, the strict inequality (\ref{3.43}) implies that
\begin{equation}\label{3.40bis}
{\rm{either}}\quad\|Dw_{\delta_0}\|_2<Z_{\delta_0}\quad{\rm{or}}
\quad \|Dw_{\delta_0}\|_2>Z_{\delta_0}\quad{\rm
if}\;\;\;\delta=\delta_0.
\end{equation}
\par  Inequalities  (\ref{3.44})  and (\ref{3.40bis}) are not a
priori estimates, since they do not imply any  bound on
$\|Dw_\delta\|_2$. Nevertheless these inequalities exclude the
closed interval $[Y_\delta^-\,,\,Y_\delta^+]$ or the point
$Z_{\delta_0}$ for $\|Dw_\delta\|_2$, and they give the hope to
prove the existence of a fixed point in the  set
$\|Dw_\delta\|_2\leq Y_\delta^-$,
 when $\delta<\delta_0$, or in the set  $\|Dw_{\delta_0}\|_2\leq Z_{\delta_0}$, when $\delta=\delta_0$.\\
 
 \vfill
\eject

\par These inequalities also explain where the two smallness conditions (\ref{A1}) and (\ref{A3}) come from.
Indeed (see Remark  \ref{rmq6.2bis}), 
these two smallness conditions imply that the value
$\delta_0$ of $\delta$ for which $\Phi_\delta$ has a double zero
satisfies $\delta_0\geq\gamma$, which is the case where, as said just above, some hope is\break
allowed.  \qquad $\Box$
%
%
} }
\end{remark}

\begin{remark}\label{rmq 3.10}
{ \rm  In the present paper, we have chosen to prove the existence
of a function $w$ which is a solution of (\ref{CCbis}) (or
in other terms of a function $w=w_{\delta_0}$ which is a solution of
(\ref{CC}) with $\delta=\delta_0$)  which satisfies
$\|w\|_{H_0^1(\Omega)}\leq Z_{\delta_0}$
 (see  (\ref{CCter})). 
When $\gamma < \delta_0$, i.e. when the inequality (\ref{A3}) is a strict inequality, 
 we could as well have chosen to prove
  that for any fixed $\delta$ with $\gamma\leq\delta < \delta_0$, there exists
   a function $\hat{w}_\delta$ which is a solution of  (\ref{CC})  which satisfies
     $\|\hat{w}_\delta\|_{H_0^1(\Omega)}\leq Y_{\delta}^-$, where $Y_\delta^-$ is the smallest zero of the function $\Phi_\delta$
     (see Remark \ref{rmq6.4} and Figure 2): indeed the proofs made in Sections 4 and 5 continue to work in this framework
     and allow one to prove this result.
\\
\indent   In this framework, if we define,
for any fixed $\delta$ with $\gamma\leq\delta < \delta_0$,
the function $\hat{u}_\delta$ by
\begin{equation}\label{3.43 bis}
\hat{u}_\delta=\dfrac{1}{\delta}\log(1+\delta|\hat{w}_\delta|)\,{\rm
sign}(\hat{w}_\delta)
\end{equation}
(compare with the definition  (\ref{3.334}) of $u$, where
$\delta=\delta_0$),  the existence of a solution $\hat{w}_\delta$ 
 of  (\ref{CC})  which satisfies
     $\|\hat{w}_\delta\|_{H_0^1(\Omega)}\leq Y_{\delta}^-$  proves
(see the equivalence Theorem \ref{theoequiv}) that
 $\hat{u}_\delta$  defined by (\ref{3.43 bis}) is a solution of (\ref{2.1}) which
  satisfies 
$$
\|e^{\delta|\hat{u}_\delta|}D\hat{u}_\delta\|_2=\|D\hat{w}_\delta\|_2\leq
Y_\delta^-.
$$
\indent But the function $u$ which is defined by (\ref{3.334})  from
the function $w$ given by Theorem \ref{theo34} satisfies (\ref{2.1})
(by the equivalence
   Theorem   \ref{theoequiv} with  $\delta=\delta_0$) and (see (\ref{4.3}) again)
$$
\|e^{\delta_0|u|}Du\|_2=\|Dw\|_2\leq Z_{\delta_0}.
$$
When $\delta<\delta_0$, we therefore have

$$
\|e^{\delta|u|}Du\|_2\leq\|e^{\delta_0|u|}Du\|_2\leq
Z_{\delta_0},\quad {\rm if}\quad\delta<\delta_0,
$$
and therefore $e^{\delta |u|}-1\in H^1_0(\Omega)$. By the
equivalence Theorem \ref{theoequiv}, the function
$\overline{w}_\delta$ defined from $u$ by
\begin{equation}\label{3.45}
\overline{w}_\delta=\dfrac{1}{\delta}(e^{\delta |u|}-1)\,{\rm
sign}(u)
\end{equation}
is  a solution of  (\ref{CC}). Moreover, since
$D\overline{w}_\delta=e^{\delta |u|}Du$ and since
$Z_{\delta_{0}}<Y_{\delta}^{+}$ for $\delta<\delta_0$ (see
(\ref{6.25})), one has in particular
$$
\|D\overline{w}_\delta\|_2<Y_{\delta}^{+},\quad{\rm if}\quad
\delta<\delta_0,
$$
which implies by (\ref{3.44}) that $\overline{w}_\delta$ satisfies

\begin{equation}\label{3.46}
\|D\overline{w}_\delta\|_2<Y_{\delta}^{-}.
\end{equation}
\par Therefore the result of Theorem \ref{theo34} (which is concerned
with the case $\delta=\delta_0$) provides us with a function $w$,
and  then with a function $u$ defined by (\ref{3.334}), and finally
with a function $\overline{w}_\delta$ defined by (\ref{3.45}) which
is a solution of
(\ref{CC}) and which satisfies (\ref{3.46}). 
This function
$\overline{w}_\delta$ is a solution $\hat{w}_\delta$ of (\ref{CC})
which satisfies $\|\hat{w}_\delta\|_{H^1_0(\Omega)}\leq Y_\delta^-$.
Therefore the result of Theorem \ref{theo34} (where
$\delta=\delta_0$) also provides us 
for every $\delta$ with $\gamma\leq\delta < \delta_0$ 
with a proof of the result stated in the
second paragraph  of the present Remark.
}
\qquad$\Box$
\end{remark}


\section{Existence of a solution for an approximate problem}
In this Section we introduce an approximation (see (\ref{care})) of problem
(\ref{CCbis}). Under the two smallness
conditions  (\ref{A1}) and (\ref{A3}), we prove  by applying
Schauder's fixed point theorem that this approximate problem has at least one
solution which satisfies the estimate (\ref{CCter}).
\\
\par Let $\delta_{0}$ be defined by (\ref{3777bis}), (\ref{Phi delta0 bis}) and
(\ref{6.15bis}).  For any $k>0$,  we consider
 the approximate problem of finding a solution $w_k$ of (compare
 with  (\ref{CCbis}))
\begin{equation}\label{care}
\left\{\begin{array}{l}
 w_{k} \in H_0^{1}(\Omega),
\\
\\
-{\rm{div}}\,(A(x)Dw_{k})
+T_{k}(K_{\delta_{0}}(x,w_{k},Dw_{k}))\,{\rm sign}_{k}(w_{k})\,=
\\
\\
=(1+\delta_0|w_{k}|)\,f(x)
+a_{0}(x)\, w_{k}+a_{0}(x)\,g_{\delta_{0}}(w_{k})\,\textrm{sign}(w_{k})\;\;\;
\textrm{in} \;\;\;\mathcal{D}'(\Omega),
\end{array}
\right.
\end{equation}
 where  $T_k$ is the  usual truncation at height $k$ defined by
 (\ref{tk}) and where  \,${\rm sign}_k:\mathbb{R}\rightarrow\mathbb{R}$ is the
approximation of the function sign which is defined by
\begin{equation}\label{def sign k}
{\textrm{sign}}_{k}(s)=\left\{
\begin{array}{lc}
ks,& {\rm if}\;\;\; |s|\leq\dfrac{1}{k}, \\
&
\\
{\textrm{sign}}(s),&{\rm if}\;\;\; |s|\geq\dfrac{1}{k}.
\end{array}
\right.
\end{equation}
\begin{theorem}\label{propo}
Assume that \rm{(\ref{2.2}), (\ref{2.3}), (\ref{2.6}), (\ref{2.5}),
(\ref{2.5bis})  and (\ref{2.4})} \it{hold true}. Assume moreover the two smallness conditions
that {\rm (\ref{A1})} and {\rm (\ref{A3})} hold true.
 Let $\delta_{0}$ be defined in Lemma \ref{lemma 2}  (see {\rm(\ref{3777bis})} and {\rm(\ref{Phi delta0 bis})}), and let $k>0$ be fixed.
\par Then there exists at least one solution of {\rm (\ref{care})} such that
\begin{equation}\label{est'}
\| w_k\|_{H_{0}^{1}(\Omega)}\leq Z_{\delta_{0}},
\end{equation}
where $ Z_{\delta_{0}}$ is defined in Lemma {\rm \ref{lemma 2}}
 (see {\rm(\ref{3777bis})}, {\rm(\ref{Phi delta0 bis})} and {\rm(\ref{6.15bis})}).
\end{theorem}
\par
The proof of Theorem  \ref{propo} consists in applying Schauder's
fixed point theorem. First we prove the two following lemmas.
\begin{lemma}\label{lem}
Assume that \rm{(\ref{2.2}), (\ref{2.3}), (\ref{2.6}), (\ref{2.5}),
(\ref{2.5bis}) and (\ref{2.4})} \it{hold true}. Let $k>0$ be fixed.
\par Then, for any $w\in H_{0}^{1}(\Omega)$, there exists
a unique solution $W$ of the following semilinear problem
\\
\begin{equation}\label{app}
\left\{\begin{array}{l} W \in H_0^{1}(\Omega),
\\
\\
-{\rm{div}}(A(x)DW)+T_{k}(K_{\delta_{0}}(x,w,Dw))\,{\rm{
sign}}_{k}(W)\,=
\\
\\
=(1+\delta_{0}|w|)\, f(x)+a_{0}(x)\, w+
a_{0}(x)\, g_{\delta_{0}}(w)\,{\rm{sign}}(w)\;\;\;{in}\;\;\;\mathcal{D'}(\Omega).
\end{array}
\right.
\end{equation}
Moreover $W$ satisfies
\begin{equation} \label{estim}
\alpha\|DW\|_{2} \leq\|f\|_{H^{-1}(\Omega)}+\delta_{0}
C_{N}^{2}\|f\|_{N/2}\|D w\|_{2}+ C_{N}^{2}\|a_{0}\|_{N/2}\|D
w\|_{2}+GC_{N}^{2+\theta}\|a_{0}\|_{q}\|D w\|_{2}^{1+\theta},
 \end{equation}
where $C_N$ and $G$ are the constants given by {\rm (2.9)} and {\rm
(\ref{G bis})}.
\end{lemma}
\begin{proof}
Problem (\ref{app}) is of the form
\begin{equation}\label{app'}
\left\{\begin{array}{l} W \in H_0^{1}(\Omega),
\\
\\
-\textrm{div}\,(A(x)DW)+\hat b(x)\,{\rm{
sign}}_{k}(W)=\widehat{F}(x)\;\;\;\rm{in }\; \;\;\mathcal{
D}'(\Omega),
\end{array}
\right.
\end{equation}
where $\hat{b}(x)$ and $\hat{F}(x)$ are given. Since $\hat
b(x)=T_{k}({K}_{\delta_{0}}(x,w,Dw))$ belongs to $ L^\infty(\Omega)$
and is nonnegative in view of (\ref{def K2}) and of $\delta_0\geq
\gamma$ (see (\ref{3777bis})), since the function ${\rm{ sign}}_{k}$ is continuous and nondecreasing, 
and since $\widehat{F}$ belong to $H^{-1}(\Omega)$ (see
e.g. the computation which allows one to obtain (\ref{4.9}) below),
this problem has a unique solution.
\\
\par Since $W \in H_0^{1}(\Omega)$, the use of $W$ as a test function in (\ref{app}) is licit.
Since $T_{k}({K}_{\delta}(x,t,\zeta))$ is nonnegative, this gives
\begin{equation}\label{detail}
\left\{\begin{array}{l} \displaystyle\int_\Omega\,A(x)\,DWDW
\,dx\leq
\\
\\
\displaystyle\leq \int_\Omega\,(1+{\delta}_{0}|w|)\,f(x)\,W
dx\,+\int_\Omega\,a_{0}(x)\, w\,W
dx\,+\int_\Omega\,a_{0}(x)\,g_{{\delta}_{0}}(w)\, {\rm{sign}}(w)\,W
dx.
\end{array}
\right.
\end{equation}
\noindent As in the computation made in Remark \ref{rmq3.9} to
obtain the 
inequality (\ref{6.15 quat}), we use in
(\ref{detail}) the coercivity (\ref{2.3}) of the matrix $A$,
H\"{o}lder's inequality with
$\dfrac{1}{\frac{N}{2}}+\dfrac{1}{2^{\star}}+\dfrac{1}{2^{\star}}=1$, inequality
(\ref{6.16})
 on $g_{{\delta}_{0}}$,
$\dfrac{1}{q}+\dfrac{1+\theta}{2^\star}+\dfrac{1}{2^\star}=1$ (which
results from the definition (\ref{b}) of $\theta$), and finally
Sobolev's inequality (\ref{2.7'}). We obtain
\begin{equation}\label{4.9}
\left\{\begin{array}{l}
 \displaystyle \alpha\|DW\|_{2}^{2}\leq
\|f\|_{H^{-1}(\Omega)}\|DW\|_2+{\delta}_{0}\|f\|_{N/2}\|w\|_{2^{\star}}\|W\|_{2^{\star}}+
\|a_0\|_{N/2}\|w\|_{2^{\star}}\|W\|_{2^{\star}}+
\\
\\

\displaystyle +
G\,\|a_0\|_{q}\|w\|_{2^{\star}}^{1+\theta}\|W\|_{2^{\star}}\leq
\\
\\
\displaystyle \leq |f\|_{H^{-1}(\Omega)}\|DW\|_2+ {\delta}_{0}
C_{N}^{2}\|f\|_{N/2}\|Dw\|_{2}\|DW\|_{2}+
C_{N}^{2}\|a_0\|_{N/2}\|Dw\|_{2}\|DW\|_{2}+
\\
\\
+\,GC_{N}^{2+\theta}\,\|a_0\|_{q}\|Dw\|_{2}^{1+\theta}\|DW\|_{2},
\end{array}
\right.
\end{equation}
which immediately implies  (\ref{estim}).
\qquad$\Box$
\end{proof}
\\
\begin{lemma}\label{lem 436}
Assume that \rm{(\ref{2.2}), (\ref{2.3}), (\ref{2.6}), (\ref{2.5}),
(\ref{2.5bis}) and (\ref{2.4})} \it{hold true}. Let $k>0$ be fixed.
\par
Let $w_n$ be a sequence such that
\begin{equation}\label{4.100}
w_n\rightharpoonup w \;\;\;{ in\; }H_{0}^{1}(\Omega){\;\;\;
weakly\;\;\;and\;\;\; a.e.\;\;\; in}\;\;\;\Omega.
\end{equation}
Define $W_n$ as the unique solution of  {\rm (\ref{app})} for
$w=w_n$, i.e.
\begin{equation}\label{continuous}
\left\{\begin{array}{l} \displaystyle W_n \in H_0^{1}(\Omega),
\\
\\
-\textrm{div}\,(A(x)DW_n)+T_{k}(K_{\delta_{0}}(x,w_n,Dw_n))\,{\rm{
sign}}_{k}(W_n)=
\\
\\
=(1+\delta_{0}|w_n|)\,f(x)+ a_{0}(x)\,w_n+
a_{0}(x)\,g_{\delta_{0}}(w_n)\,{\rm{sign}}(w_n)\;\;\;\textrm{ in
}\;\;\;\mathcal{D'}(\Omega).
\end{array}
\right.
\end{equation}
Assume moreover that for a subsequence, still denoted by $n$, and for some $W^\star\in H_{0}^{1}(\Omega)$, $W_n$
satisfies
\begin{equation}\label{4.101}
W_n\rightharpoonup W^\star \;\;\;{\rm in\;\;\;
}H_{0}^{1}(\Omega)\;\;\;{ weakly}\;\;\;{ and\;\;\; a.e.\;\;\;
in}\;\;\;\Omega.
\end{equation}
\par Then for the same subsequence one has
\begin{equation}\label{4.13bis}
W_n\rightarrow W^\star {\;\;\; in\;\;\; }H_{0}^{1}(\Omega){\;\;\;
strongly}.
\end{equation}
\end{lemma}
\proof Since $ W_n- W^\star \in H_{0}^{1}(\Omega)$,  the use of
$(W_n-W^{\star})$ as test function in (\ref{continuous}) is licit.
This gives

\begin{equation}\label{411b}
\left\{\begin{array}{l}
\displaystyle\int_{\Omega}\,A(x)D(W_n-W^{\star})D(W_n-W^{\star})\,dx=
\\
\\
\displaystyle =
-\int_{\Omega}\,A(x)DW^{\star}D(W_n-W^{\star})\,dx\,+
\\
\\
\displaystyle
-\int_{\Omega}\,T_{k}\left(K_{\delta_{0}}(x,w_n,Dw_n)\right)\,\textrm{sign}_k(W_n)\,(W_n-W^{\star})\,dx\,+
\\
\\
\displaystyle
+\int_{\Omega}\,(1+\delta_0|w_n|)f(x)\,(W_n-W^{\star})\,dx
\,+\int_{\Omega}\,a_{0}(x)\,w_n (W_n-W^{\star})\,dx\,+
\\
\\
\displaystyle +\int_{\Omega}\;a_{0}(x)\, g_{\delta_0}(w_n)
{\,\textrm{sign}}(w_n)\,(W_n-W^{\star})\,dx.
\end{array}
\right.
\end{equation}
\par We claim that every term of the right-hand of (\ref{411b}) tends to zero as $n$ tends to infinity.
\par For the first term, we just use the fact that $W_n-W^{\star}$ tends to zero in $H_{0}^{1}(\Omega)$ weakly.
\par For the second term,  we use the fact $T_{k}(K_{\delta_0} (x,w_n,Dw_n))\,{\rm sign}_k(W_n)$ is bounded in
$L^{\infty}(\Omega)$,  since $k$ is fixed, while $W_n- W^{\star}$
tends to zero in $L^{1}(\Omega)$ strongly.
\par For the last three terms we observe that, 
since $w_n$ and $W_n$  respectively converge almost everywhere to $w$ and to $W^*$ 
(see (\ref{4.100}) and (\ref{4.101})), we have
\begin{equation}
\left\{\begin{array}{ll}\label{4.15}
 (1+\delta_0 |w_{n}|)\, f(x)\, (W_{n}- W^{\star}) \rightarrow 0 & \;\;\;{\rm{a.e.\;\;\; in }}\;\;\;\Omega,\\
 &  \\

a_{0}(x)\,w_{n}\,(W_{n}- W^{\star}) \rightarrow 0 & \;\;\;{\rm{a.e.\;\;\; in }}\;\;\;\Omega, \\
 &  \\
a_{0}(x)\,g_{\delta_0}(w_{n})\,{\rm{sign}}(w_{n})\,(W_{n}-
W^{\star})  \rightarrow 0  &\;\;\;{\rm{a.e.\;\;\; in }}\;\;\;\Omega.
 \end{array}
\right.
\end{equation}
\\
\par We will now prove that each of the three sequences which appear
in (\ref{4.15}) are equiintegrable. Together with (\ref{4.15}),
 this will imply that these sequences converge to zero in $L^{1}(\Omega)$ strongly,  and this will prove that the three
 last terms of the right-hand side of
  (\ref{411b}) tend to zero as $n$ tends to infinity.
  
\vfill
\eject  
  
\par In order to prove that  the sequence $(1+\delta_0 |w_{n}|)\,f(x)\,(W_{n}-
W^{\star})$ is  equiintegrable, we use H\"{o}lder's
inequality with $\dfrac{1}{2^\star} + \dfrac{1}{\frac{N}{2}}+\dfrac{1}{2^\star}=1$. For
any measurable set $E$, $E\subset\Omega$, we have
\begin{equation*}
\left\{\begin{array}{l}
\displaystyle\int_{E}\,|(1+\delta |w_{n}|)\,f(x)\,\,(W_{n}- W^{\star})|\,dx\leq \\
\\
\leq
\displaystyle\left(\int_{E}\,|f(x)|^{N/2}dx\right)^{2/N}\,\|(1+\delta_0
|w_{n}|)\|_{2^{\star}}\,\|W_{n}- W^{\star}\|_{2^{\star}}
 \leq c\left(\int_{E}\,|f(x)|^{N/2}\,dx\right)^{2/N},
 \end{array}
 \right.
\end{equation*}
\noindent where $c$ denotes a constant which is independent of $n$.
\\
\indent Proving that the sequence $a_{0}(x)\,w_{n}\,(W_{n}-
W^{\star})$ is  equiintegrable is similar, since for
any measurable set $E$, $E\subset\Omega$, we have
\begin{equation*}
\left\{\begin{array}{l}
\displaystyle\int_{E}\,|a_{0}(x)\,w_{n}\,(W_{n}- W^{\star})|\,dx
\,\leq
\\
\\
\leq\displaystyle\left(\int_{E}\,|a_{0}(x)|
^{N/2}\,dx\right)^{2/N}\,\|w_{n}\|_{2^{\star}}\,\|W_{n}-
W^{\star}\|_{2^{\star}}
 \displaystyle\leq
 c\left(\int_{E}\,|a_{0}(x)|^{N/2}\,dx\right)^{2/N}.
 \end{array}
 \right.
\end{equation*}
\\
\indent  Finally, in order to prove that the sequence
 $a_{0}(x)\,g_{\delta_0}(w_{n})\,{\rm{sign}}(w_{n})\,(W_{n}- W^{\star})$  is  equiintegrable,
 we use as in (\ref{4.9}) inequality (\ref{6.16}) and H\"{o}lder's inequality with $\dfrac{1}{q}+\dfrac{1+\theta}{2^{\star}}
 +\dfrac{1}{2^{\star}}=1$; for any measurable set $E$, $E\subset\Omega$, we have
\begin{equation*}
\left\{\begin{array}{l}
\displaystyle\int_{E}\,|a_{0}(x)\,g_{\delta_0}(w_{n})\,{\rm{sign}}(w_{n})\,(W_{n}-
W^{\star})|\,dx\leq
\int_{E}\,|a_{0}(x)| \, G \, |w_{n}|^{1+\theta} \, |W_{n}- W^{\star}|\;dx\;\leq\\
\\
\leq
\displaystyle\left(\int_{E}\,|a_{0}(x)|^{q}\,dx\right)^{1/q}G\,\|w_{n}\|^{1+\theta}_{2^{\star}}\,\|W_{n}-
W^{\star}\|_{2^{\star}}\leq c\left(\int_{E}\,|a_{0}(x)|^{q}\,dx\right)^{1/q}.\\
  \end{array}
 \right.
\end{equation*}
\\
\indent We have proved that the right-hand side of (\ref{411b})
tends to zero. Since the matrix $A$ is coercive (see (\ref{2.3})),
this proves that $W_{n}$ tends to $W^{\star}$ in $H_{0}^{1}(\Omega)$
strongly. Lemma  \ref{lem 436} is proved.
\qquad$\Box$
\\
\\
\noindent
\begin{proof}[Proof of Theorem \ref{propo}]
Recall that  in this Theorem $k>0$ is fixed.
\\
 \indent Consider the ball $B$ of
$H_{0}^{1}(\Omega)$ defined by
\begin{equation}\label{B}
B =\{w\in H_{0}^{1}(\Omega):\|Dw\|_{2}\leq Z_{\delta_{0}}\},
\end{equation}
where $Z_{\delta_{0}}$ is defined from $\delta_{0}$  by
(\ref{6.15bis}).
\par Consider also the mapping
$S:H_0^{1}(\Omega)\rightarrow H_0^{1}(\Omega)$ defined by
\begin{equation}\label{S}
S(w)=W,
\end{equation}
where for every $w\in H_0^1(\Omega)$, $W$ is the unique solution of
(\ref{app}) (see Lemma \ref{lem}) .
\\
\par We will apply Schauder's fixed point theorem in the Hilbert space $H_{0}^{1}(\Omega)$ to the mapping  $S$ and to the ball $B$.
\\
\\
\noindent{\textbf{First step}} 
In this step we prove that $S$ maps $B$  into itself.
\par Indeed by Lemma \ref{lem}, $W=S(w)$ satisfies (\ref{estim});
therefore, when $\|Dw\|_{2}\leq Z_{{\delta}_{0}}$, one has, in view
of the definition (\ref{Phidelta bis}) the function $\Phi_{\delta}$
and of the property (\ref{Phi delta0 bis}) of $Z_{\delta_0}$,
\begin{equation}
\left\{\begin{array}{l} \alpha\|DW\|_{2} \leq
\\
\\
\leq \|f\|_{H^{-1}(\Omega)}+\delta_{0} C_{N}^{2}\|f\|_{N/2}\|D
w\|_{2}+ C_{N}^{2}\|a_{0}\|_{N/2}\|D
w\|_{2}+GC_{N}^{2+\theta}\|a_{0}\|_{q}\|D w\|_{2}^{1+\theta}\leq
\\
\\
\leq \|f\|_{H^{-1}(\Omega)}+\delta_{0}
C_{N}^{2}\|f\|_{N/2}\,Z_{{\delta}_{0}} +
C_{N}^{2}\|a_{0}\|_{N/2}Z_{{\delta}_{0}}+GC_{N}^{2+\theta}\|a_{0}\|_{q}\,Z_{{\delta}_{0}}^{1+\theta}
\,=
\\
\\
=\alpha Z_{{\delta}_{0}} +\Phi_{\delta_{0}}(Z_{{\delta}_{0}})=\alpha
Z_{{\delta}_{0}},
 \end{array}
\right.
\end{equation}
i.e. $\|DW\|_{2}\leq Z_{{\delta}_{0}}$, or in other terms $W\in B$,
which proves that $S(B)\subset B$.
\\
\\
\noindent{\textbf{Second step}}  
In this step we prove that $S$ is
continuous from $H_{0}^{1}(\Omega)$ strongly into
$H_{0}^{1}(\Omega)$ strongly.
\par  For this we consider  a sequence such that
\begin{equation}\label{4.137}
w_n\in B,\;\;\;w_n\rightarrow w\;\;\;{\rm
in}\;\;\;H_{0}^{1}(\Omega)\;\;\;{\rm strongly},
\end{equation}
and we define $W_n$ as $W_n=S(w_n)$, i.e. as the solution of
(\ref{continuous}).
\par
The functions $w_n$ belong to $B$, and therefore the functions $W_n$
belong to $B$ in view of the first step. We can therefore extract a
subsequence, still denoted by $n$, such that for some $W^\star\in
H_{0}^{1}(\Omega)$,
\begin{equation}\label{4.139}
W_n\rightharpoonup W^\star {\;\;\;\rm in\;
}H_{0}^{1}(\Omega){\;\;\;\rm weakly}\;\;\;{\rm and\;\;\; a.e.\;\;\;
in}\;\;\;\Omega.
\end{equation}
\noindent We can moreover assume that for a further subsequence, still denoted by $n$,
\begin{equation}\label{4.139bis}
w_n\rightharpoonup w {\;\;\;\rm a.e.\;\;\;in}\;\;\; \Omega\;\;\;{\rm
and } \;\;\;Dw_n\rightharpoonup Dw {\;\;\;\rm a.e.\;\;\;in}\;\;\;
 \Omega.
\end{equation}
 Since the assumptions of Lemma \ref{lem 436} are satisfied
by the subsequences $w_n$ and $W_n$, the subsequence $W_n$ converges
to $W^\star$ in $H^1_0 (\Omega)$ strongly.
 \par We  now pass to the limit in equation (\ref{continuous})  as $n$ tends to infinity by using the fact that ${\textrm{sign}}_{k}(s)$
 and $g_{\delta_0}(s)\,{\rm{sign}}(s)$ are Carath\'{e}odory functions, and the first result of (\ref{consconv}) as far
  as $T_{k}(K_{\delta_0}(x,w_n,Dw_n))$ is concerned (this point is the only point of the proof of Theorem \ref{propo}
  where the assumption of strong $H^1_0 (\Omega)$ convergence in (\ref{4.137}),
  or more exactly its consequence (\ref{4.139bis}), is used).
This implies that $W^{\star}$ is
  a solution of (\ref{app}). Since the solution of (\ref{app}) is unique, one has  $W^{\star}=S(w)$.
\par In view of the fact that $W^{\star}$ is uniquely determined, we conclude that it was not necessary to extract a
 subsequence in (\ref{4.139}) and (\ref{4.139bis}), and that the whole sequence $W_n=S(w_n)$ converges in $H_{0}^{1}(\Omega)$ strongly to
  $W^{\star}=S(w)$. This proves the continuity of the application $S$.
\\
\\
\noindent {\textbf{Third step}} 
In this step we prove that $S(B)$ is
precompact in $H_{0}^{1}(\Omega)$.
\par For this we consider a sequence $w_n\in B$ and we define $W_n$ as $W_n=S(w_n)$;  in other terms $W_n$ is the
solution of (\ref{continuous}). Since $w_n$ and  $W_n$ belong to
$B$, they are bounded in $H_{0}^{1}(\Omega)$, and
 we can extract a subsequence, still denoted by $n$,  such that
$$ w_n\rightharpoonup w\;\;\;\textrm{ in }\;\;\;H_0^{1}(\Omega) \;\;\;{\rm weakly\; \;\;and\;\;\; a.e.\;\;\; in } \;\;\;\Omega,$$
$$W_n\rightharpoonup W^{\star}\;\;\;\textrm{ in }\;\;\;H_0^{1}(\Omega) \;\;\;{\rm weakly \;\;\; and\;\;\; a.e.\;\;\; in }\;\;\; \Omega .$$
Since $w_n$ and  $W_n$  satisfies the assumptions of Lemma \ref{lem
436}, we have
$$W_n\rightarrow W^\star\;\;\; {\rm in }\;\;\;H_{0}^{1}(\Omega)\;\;\; {\rm strongly}.$$
\par This proves that $S(B)$ is precompact in $H_{0}^{1}(\Omega)$ (note that in contrast
with the second step, we do not need here to prove that
$W^\star=S(w)$).
\\
\\
\noindent {\textbf{End of the proof of Theorem \ref{propo}}} We have
proved that the application $S$ and the ball $B$ satisfy the
assumptions of Schauder's fixed point theorem. Therefore there
exists at least one  $w_k\in B$ such that $S(w_k)=w_k$. This proves
Theorem \ref{propo}.
\qquad$\Box$
\end{proof}

\section{Proof of Theorem~\ref{theo34}}
Theorem \ref{propo} asserts that for for every $k>0$ fixed there
exists at least a solution $w_k$ of (\ref{care}) which satisfies
(\ref{est'}). We can therefore extract a subsequence, still denoted
by $k$, such that for some $w^\star\in H^{1}_{0}(\Omega)$
\begin{equation}\label{5.1}
 w_k\rightharpoonup w^{\star} \quad {\textrm{ in }}\;\;\; H_{0}^{1}(\Omega) \;\;\; {\rm weakly\;\;\; and\;\;\; a.e.\;\;\; in } \;\;\;\Omega,\\
\end{equation}
where $w^\star$ satisfies
\begin{equation}\label{5.1bis}
\|w^\star\|_{H_{0}^{1}(\Omega)}\leq Z_{\delta_0}, 
\end{equation}
i.e.  (\ref{CCter}).
\\ 
\par In this Section we will first prove that for this subsequence
\begin{equation}\label{5.2}
  w_k\rightarrow w^\star \;\;\; {\rm in}\;\;\; H_{0}^{1}(\Omega)\;\;\;{\rm strongly},
\end{equation}
and then that $w^\star $ is a solution of (\ref{CCbis}) 
(which satisfies 
(\ref{CCter})). This will prove Theorem \ref{theo34}.
\\
\par To prove (\ref{5.2}), we use a technique which traces back
to \cite{bens} (see also \cite{fm}).
\par For $n>0$, we define $G_n:\mathbb{R}\rightarrow \mathbb{R}$ as the remainder of
the truncation at height $n$, namely
\begin{equation}\label{5.3bis}
G_n (s)=s-T_n (s),\;\;\;\forall s\in\mathbb{R}
\end{equation}
where $T_n$ is the truncation at height $n$ defined by (\ref{tk}),
or in other terms
\begin{equation}\label{gn}
G_{n}(s)= \left\{\begin{array}{ll}
s+n &  \;\;\;{\rm if }\;\;\; s\leq -n ,\\
0 &  \;\;\;{\rm if }\;\;\; -n\leq s\leq n,\\
s-n & \;\;\; {\rm if }\;\;\; s\geq n.
\end{array}
\right.
\end{equation}
\\
\par First we  prove the two following Lemmas.

\begin{lemma}\label{lem5.1}
Assume that \rm{(\ref{2.2}), (\ref{2.3}), (\ref{2.6}), (\ref{2.5}),
(\ref{2.5bis})   and (\ref{2.4})} \it{hold true}. Assume moreover
that the two smallness conditions {\rm (\ref{A1})} and {\rm (\ref{A3})} hold true. Let $w_k$ be a
solution of {\rm (\ref{care})}. Assume finally that the subsequence $w_k$
satisfies {\rm (\ref{5.1})}.
\\
\indent Then for this subsequence we have
\begin{equation}\label{CV1}
\displaystyle \underset{k\rightarrow
+\infty}\limsup\int_{\Omega}\,|DG_n (w_k)|^{2}dx\rightarrow 0 \;\;\;
{ as }\;\;\;n\rightarrow +\infty.
\end{equation}
\end{lemma}
\begin{proof}
Since $G_n (w_k) \in H_{0}^{1}(\Omega)$, the use of $G_n (w_k)$ as
test function in (\ref{care}) is licit. This gives
\begin{equation}\label{5.111}
\left\{\begin{array}{l}
\displaystyle\int_{\Omega}\,A(x)Dw_{k}DG_n(w_{k})\,dx+\displaystyle\int_{\Omega}\,T_{k}(K_{\delta_0}
(x,w_{k},Dw_{k}))\,\textrm{sign}_k(w_{k})\,G_n(w_{k})\,dx=
\\
\\
 =\displaystyle\int_{\Omega}\,\Big((1+\delta_0|w_{k}|)\,f(x)+a_{0}(x)\,w_{k}+a_{0}(x)\, g_{\delta_0}(w_{k})
\,\textrm{sign}(w_{k})\Big)\,G_n(w_{k})\,dx.
\end{array}
\right.
\end{equation}
\indent Using the coercivity (\ref{2.3}) of the matrix $A$, we have
for the first term of (\ref{5.111})
\begin{equation}\label{5.112}
\displaystyle\int_{\Omega}\,A(x)Dw_{k}DG_n(w_{k})\,dx=\int_{\Omega}\,A(x)DG_n(w_{k})DG_n(w_{k})\,dx\geq
\alpha\int_{\Omega}\,|DG_n(w_{k}|^{2}\,dx.
\end{equation}
\noindent On the other hand, since $${\rm sign}_{k}(s)\,G_n(s)=|{\rm
sign}_{k}(s)| \,|G_n(s)|\geq 0,\qquad \forall s\in \mathbb{R},$$
 and since $T_{k}(K_{\delta_0}(x,w_{k},Dw_{k}))\geq 0$ in view of (\ref{def K2}) and of $\delta_0 \geq \gamma$ (see (\ref{3777bis})), 
we have
\begin{equation}\label{5.113b}
\displaystyle\displaystyle\int_{\Omega}\,T_{k}(K_{\delta_0}(x,w_{k},Dw_{k}))\,\textrm{sign}_k(w_{k})\,G_n(w_{k})\,dx\geq
0.
\end{equation}
\noindent Finally, we observe that, by a proof which is similar to
the one that we used in the proof of Lemma \ref{lem 436}, we have
 \begin{equation}\label{5.114b}
 \left\{\begin{array}{l}
\Big((1+\delta_0|w_{k}|)\,f(x)+a_{0}(x)\,w_{k}+a_{0}(x)\, g_{\delta_0}(w_{k})
\,\textrm{sign}(w_{k})\Big)\,G_n(w_{k})\rightarrow \\
\\
   \rightarrow\Big((1+\delta_0|w^\star|)\,f(x)+a_{0}(x)\,w^\star+a_{0}(x)\, g_{\delta_0}(w^\star)
\,\textrm{sign}(w^\star)\Big)\,G_n(w^\star)\\
\\
{\rm in}\;\;\;L^1(\Omega)\;\;\;{\rm strongly},
 \end{array}
 \right.
\end{equation}
since the functions in the left-hand side of (\ref{5.114b}) converge
almost everywhere in $\Omega$ and are equi\-integrable.
\par Together with (\ref{5.111}), the three results (\ref{5.112}), (\ref{5.113b}) and
(\ref{5.114b}) imply that
\begin{equation}\label{5.115}
\left \{\begin{array}{l}
   \underset{k\rightarrow +\infty}\limsup \;\alpha\displaystyle\int_{\Omega}\,|DG_n(w_{k}|^{2}\,dx \leq\\
   \\
 \leq\displaystyle\int_{\Omega}\Big( (1+\delta_0|w^\star|)\,f(x)+ a_{0}(x)\,w^\star+a_{0}(x)\, g_{\delta_0}(w^\star)
\,\textrm{sign}(w^\star)\Big)\,G_n(w^\star)\,dx.
 \end{array}
\right.
\end{equation}
\par But since $|G_n(w^\star)|\leq |w^\star|$ and since $G_n(w^\star)=0$ in the set $\{|w^\star|\leq n\}$, the right-hand side of (\ref{5.115})
 is bounded from above by
\begin{equation}
\displaystyle\int_{\{|w^\star|\geq
n\}}\,\Big((1+\delta_0|w^\star|)\,|f(x)|+
a_{0}(x)\,|w^\star|+a_{0}(x)\, |g_{\delta_0}(w^\star)| \Big)
\,|w^\star|\,dx,
\end{equation}
 which tends to zero when $n$ tends to infinity because the integrand belongs to $L^1(\Omega)$.\\
\indent This prove (\ref{CV1}). 
\qquad $\Box$
\end{proof}
\\
\begin{lemma}\label{lem5.2}
Assume that \rm{(\ref{2.2}), (\ref{2.3}), (\ref{2.6}), (\ref{2.5}),
(\ref{2.5bis})   and (\ref{2.4})} \it{hold true}. Assume moreover
that {the two smallness conditions \rm (\ref{A1})} and {\rm (\ref{A3})} hold true. Let $w_k$ be a
solution of {\rm (\ref{care})}. Assume finally that the subsequence $w_k$
satisfies {\rm (\ref{5.1})}.
\\
\indent Then for this subsequence we have for every $n>0$ fixed
\begin{equation}\label{CV2}
T_n(w_k)\rightarrow T_n(w^\star)\;\;\;{ in }\;\;\;
H^{1}_{0}(\Omega)\;\;\;{ strongly\;\;\; as}\;\;\;k\rightarrow
+\,\infty .
\end{equation}
\end{lemma}
\begin{proof}
In this proof $n$ is fixed. We define
\begin{equation}\label{7.0}
z_k=T_n(w_k)-T_n(w^\star),
\end{equation}
and we fix a $C^1$ function $\psi:\mathbb{R}\rightarrow\mathbb{R}$
such that
\begin{equation}\label{7.000}
\psi(0)=0,\quad \psi'(s)-(c_0+\delta_0)\,|\psi(s)|\geq1/2,\;\;\;
\forall s\in \mathbb{R},
\end{equation}
where $c_0$ is the constant which appears in the left-hand side of assumption
(\ref{2.6}) on $H$; there exist such functions $\psi$: indeed an example is
$$
\psi(s)=s \exp\left(\dfrac{(c_0+\delta_0)^2}{4}s^2\right).
$$
\noindent{\textbf{First step}} 
Since $z_k\in H^{1}_{0}(\Omega)\cap
L^\infty(\Omega)$, and since $\psi(0)=0$, the function $\psi(z_k)$
belongs to\break $H^{1}_{0}(\Omega)\cap L^\infty(\Omega)$. The use
of $\psi(z_k)$ as test function in (\ref{care}) is therefore licit.
This gives

\begin{equation}\label{psi1}
\left\{\begin{array}{l} \displaystyle\int_{\Omega}\,A(x)Dw_{k}
Dz_k\,\psi'(z_k)dx
+\int_{\Omega}\,T_{k}(K_{\delta_{0}}(x,w_{k},Dw_{k}))\, {\rm
sign}_{k} (w_{k})\,\psi(z_{k})\,dx\, =
\\
\\
=\displaystyle\int_{\Omega}\,\Big((1+\delta_0
|w_{k}|)\,f(x)+a_{0}(x)\,w_{k}+a_{0}(x)\, g_{\delta_0}(w_{k})
\,\textrm{sign}(w_{k})\Big)\,\psi(z_k)\,dx.
\end{array}
\right.
\end{equation}

\par Since
\begin{equation}\label{5.9}
  \begin{array}{l}
  Dw_k =DT_n(w_k)+DG_n(w_k) = Dz_k+DT_n(w^\star)+DG_n(w_k),
  \end{array}
  \end{equation}
the first term of the left-hand side of (\ref{psi1}) reads as
\begin{equation}\label{psi2}
\left\{\begin{array}{l}
\displaystyle\int_{\Omega}\,A(x)Dw_{k}Dz_k\,\psi'(z_k)\,dx=
\displaystyle\int_{\Omega}\,A(x)Dz_{k}{Dz_k}\,\psi'(z_k)\,dx\,+
\\
\\
\displaystyle
+\int_{\Omega}\,A(x)DT_n(w^\star)Dz_k\,\psi'(z_k)\,dx\,
 +\int_{\Omega}\,A(x)DG_n(w_k)Dz_k\,\psi'(z_k)\,dx.
\end{array}
\right.
\end{equation}

\par On the other hand, splitting $\Omega$ into $\Omega=\{|w_k|>n\}\cup\{|w_k|\leq n\}$, the second term of the left-hand side
of (\ref{psi1}) reads as
\begin{equation}\label{5.12}
\left \{\begin{array}{l}
\displaystyle\int_{\Omega}\,T_{k}(K_{\delta_0}(x,w_{k},Dw_{k}))\,\textrm{sign}_k(w_{k})\,\psi(z_k)\,dx\,=  \\
   \\
= \displaystyle\int_{\{|w_k|>n\}}\,T_{k}(K_{\delta_0}(x,w_{k},Dw_{k}))\,\textrm{sign}_k(w_{k})\,\psi(z_k)\,dx\,+ \\
  \\
 + \displaystyle\int_{\{|w_k|\leq n\}}\,T_{k}(K_{\delta_0}(x,w_{k},Dw_{k}))\,\textrm{sign}_k(w_{k})\,\psi(z_k)\,dx.
\end{array}
\right.
\end{equation}

For what concerns the first term of the right-hand of (\ref{5.12}),
we claim that
\begin{equation}\label{5.15}
\begin{array}{l}
 \displaystyle\int_{\{|w_k|>n\}}\,T_{k}(K_{\delta_0}(x,w_{k},Dw_{k}))\,\textrm{sign}_k(w_{k})\,\psi(z_k)\,dx\geq 0 \, ;
\end{array}
\end{equation}
indeed in $\{|w_k|>n\}$, the integrand is nonnegative since on the
first  hand $T_{k}(K_{\delta_0}(x,w_{k},Dw_{k}))\geq 0$  in view of
(\ref{def K2}) and  of $\delta_0\geq\gamma$ (see (\ref{3777bis})), 
and since on the other
hand one has
 \begin{equation}\label{5.3636}
\begin{array}{l}
\textrm{sign}_k(w_{k})\,\psi(z_k)\geq 0\;\;\; {\rm
in}\;\;\;\{|w_k|>n\};
\end{array}
\end{equation}
indeed since sign($s$) and sign$_k(s)$ have the same sign, it is
equivalent either to prove (\ref{5.3636}) or to prove that
 \begin{equation}\label{5.3636bis}
\begin{array}{l}
\textrm{sign}(w_{k})\,\psi(z_k)\geq 0\;\;\; {\rm
in}\;\;\;\{|w_k|>n\};
\end{array}
\end{equation}
 but in $\{|w_k|>n\}$ one has $z_k=n \,{\rm sign}(w_k)-T_n(w^\star)$, and therefore
${\rm sign}(z_k)={\rm sign}(w_k)$; this implies that
$$
{\rm sign}(w_k)\,\psi(z_k)={\rm
sign}(z_k)\,\psi(z_k)=|\psi(z_k)|\;\;\;{\rm in}\;\;\;\{|w_k|>n\},
$$
which proves (\ref{5.3636bis}).
\par For what concerns the second term of the right-hand side of (\ref{5.12}), we observe that, in view of (\ref{def K2}) and of
 $\delta_0\geq\gamma$ (see( \ref{3777bis})), we have
\begin{equation}
\left \{\begin{array}{l}
|T_{k}(K_{\delta_0}(x,w_{k},Dw_{k}))\,\textrm{sign}_k(w_{k})\,\psi(z_k)|
 \leq|K_{\delta_0}(x,w_{k},Dw_{k})||\psi(z_k)|\leq
\\\\
 \leq (c_0
+\delta_0)|\psi(z_k)|\,A(x)Dw_kDw_k.
\end{array}
\right.
\end{equation}
Since in view of (\ref{5.9}) one has
 $$
  Dw_k = Dz_k+DT_n(w^\star)\;\;\;{\rm in}\;\;\;\{|w_k|\leq n\},
  $$
we obtain
%
\begin{equation}\label{5.1616}
\left \{\begin{array}{l} \displaystyle\int_ {\{|w_k|\leq n\}}
\,T_{k}(K_{\delta_0}(x,w_{k},Dw_{k}))\,\textrm{sign}_k(w_{k})\,\psi(z_k)
\,dx\geq
   \\
   \\
\geq - \displaystyle\int_ {\{|w_k|\leq n\}}\,
(c_0+\delta_0)|\psi(z_k)|\, A(x)Dw_kDw_k \,dx =
\\
\\
= - \displaystyle\int_ {\{|w_k|\leq n\}} \,
(c_0+\delta_0)|\psi(z_k)|\, A(x)(Dz_k +DT_n(w^\star))(Dz_k
+DT_n(w^\star))\,dx \geq
\\
\\
\geq - \displaystyle\int_ {\Omega}\, (c_0+\delta_0)|\psi(z_k)|\,
A(x)(Dz_k +DT_n(w^\star))(Dz_k +DT_n(w^\star))\,dx \geq
\\
\\
\geq - \displaystyle\int_ {\Omega}\, (c_0+\delta_0)|\psi(z_k)|\,
A(x)Dz_kDz_k \,dx\, +
\\
\\
- \displaystyle\int_ {\Omega}\,
(c_0+\delta_0)|\psi(z_k)|\\
\\
\qquad\Big(A(x)DT_n(w^\star)Dz_k+A(x)Dz_k\,DT_n(w^\star)
+A(x)DT_n(w^\star)DT_n(w^\star)\Big)\,dx.
\end{array}
\right.
\end{equation}
From (\ref{psi1}), (\ref{psi2}), (\ref{5.12}), (\ref{5.15}) and
(\ref{5.1616}) we deduce that
\begin{equation}\label{psi6}
\left\{\begin{array}{l} \displaystyle\int_{\Omega}\,A(x)Dz_k
Dz_k\,\big(\psi'(z_k)-(c_0+\delta_0)|\psi (z_k)|\big)\,dx \,\leq
\\
\\
\leq\displaystyle -\int_{\Omega}\,A(x)DT_n(w^\star)Dz_k\,\psi
'(z_k)dx-\int_{\Omega}\,A(x)DG_n(w_k)Dz_k\,\psi '(z_k)\,dx\, +
\\\\
+\displaystyle \int_{\Omega}\,(c_0+\delta_0)|\psi (z_k)|
\\
\\
\qquad\Big ( A(x)DT_n(w^\star)Dz_k+A(x)Dz_kDT_n(w^\star) +
A(x)DT_n(w^\star)DT_n(w^\star)\Big)\,dx\,+
\\
\\
\displaystyle
+\int_{\Omega}\,\Big((1+\delta_0|w_{k}|)\,f(x)+a_{0}(x)\,w_{k}+a_{0}(x)\, g_{\delta_0}(w_{k})
\,\textrm{sign}(w_{k})\Big)\,\psi(z_k)\,dx.
\end{array}
\right.
\end{equation}
\\
\noindent{\textbf{Second step}} 
We claim that each term of the right-hand side
of (\ref{psi6})  tends to zero as $k$ tends to infinity. Since
$\psi'(z_k)-(c_0+\delta)\,|\psi(z_k)|\geq1/2$ by (\ref{7.000}), and
since the matrix $A$ is coercive (see (\ref{2.3})), this will imply
that
$$
z_k\rightarrow 0\;\;\;{\rm in\;\;\;}H^{1}_{0}(\Omega)\;\;\;{\rm
strongly},
$$
or in other terms (see the definition (\ref{7.0}) of $z_k$) that
$$
T_n(w_k)\rightarrow T_n(w^\star)\;\;\;{\rm in
}\;\;\;H^{1}_{0}(\Omega)\;\;\; {\rm
strongly\;\;\;as\;\;\;}k\rightarrow +\infty,
$$
which is nothing but (\ref{CV2}). Lemma \ref{lem5.2} will therefore
be proved whenever the claim will be proved.
\\
\par
In order to prove the claim let us recall that in view of
(\ref{5.1}) and of the definition (\ref{7.0}) of $z_k$ one has
$$
z_k\rightharpoonup 0 \;\;\;{\rm in }\;\;\;H^{1}_{0}(\Omega)\;\;\;
{\rm  weakly},\;\;\; L^{\infty}(\Omega)\;\;\;{\rm weakly} \; {\rm star}
\;\;\;{\rm and \;\;\; a.e.}\;\;\;{\rm in}\;
\;\;\Omega\;\;\;{\rm as}\;\;\; k\rightarrow +\infty.
$$
\par Since $\psi(0)=0$, this implies that $\psi(z_k)$ tends to zero
almost everywhere in $\Omega$ and in $L^\infty(\Omega)$ weakly star
as $k$ tends to infinity, which in turn implies that
$$
Dz_k\,\psi'(z_k)=D\psi(z_k)\rightharpoonup 0\;\;\;{\rm in
}\;\;\;L^2(\Omega)^N\;\;\;{\rm weakly}\;\;\;{\rm
as}\;\;\;k\rightarrow +\infty.
$$
This implies that the first term of the right-hand side of
(\ref{psi6}) tends to zero as $k$ tends to infinity.
\\
\par For the second term of the right-hand side of of (\ref{psi6}) we observe that
$$
A(x)DG_n(w_k)Dz_k=A(x)DG_n(w_k)(DT_n(w_k)-DT_n(w^\star)) =-
A(x)DG_n(w_k)DT_n(w^\star),
$$
and that by Lebesgue's dominated convergence theorem
$$
DT_n(w^\star)\,\psi'(z_k)\rightarrow DT_n(w^\star)\,\psi'(0)
\;\;\;{\rm in }\;\;\;L^2(\Omega)^N\;\;\;{\rm strongly}\;\;\;{\rm
as}\;\;\; k\rightarrow +\infty,
$$
while $ DG_n(w_k)$ tends to $DG_n(w^\star)$ weakly in
$L^2(\Omega)^N$.
 Since $A(x)DG_n(w^\star)DT_n(w^\star)=0$ almost everywhere, the second term of the right-hand side of  (\ref{psi6}) tends to zero.\\

\par For the third term of the right-hand side of  (\ref{psi6}), we observe that
$$(c_0+\delta_0)|\psi (z_k)| A(x)DT_n(w^\star)\rightarrow 0\;\;\;{\rm in }\;L^2(\Omega)^N\;\;\;{
\rm strongly}\;\;\;{\rm as}\;\;\;k\rightarrow +\infty$$ by
Lebesgue's dominated convergence theorem,  since $\psi(z_k)$ is
bounded in $L^\infty(\Omega)$ and since $\psi(z_k)$ tends
  almost everywhere to zero because $\psi(0)=0$.
 Since $Dz_k$ is bounded
in $L^2(\Omega)^N $, this implies that the first part of this third term tends to zero. A similar proof holds true for the two others parts of this third term.\\
\par Finally the fourth term of the right-hand side of  (\ref{psi6}) tends to zero by
a proof which is similar to the one that we used in the proof of
Lemma \ref{lem 436}, since the integrand converges almost everywhere
to zero and is equiintegrable.
\\
\par The claim made at the beginning of the second step is proved.
This completes the proof of Lemma \ref{lem5.2}.
\qquad$\Box$
\end{proof}
\\
\\

\noindent\textbf{End of the proof of Theorem \ref{theo34}}
\\
\\
\noindent{\textbf{First step}} 
Since we have
$$w_k-w^\star=T_n(w_k)+G_n(w_k)-T_n(w^\star)-G_n(w^\star), $$
and since by Lemma \ref{lem5.2} (see (\ref{CV2})) we have
$$
\|T_n(w_k)-T_n(w^\star)\|_{H_{0}^{1}(\Omega)}\rightarrow 0\quad {\rm
as}\;\;\;k\rightarrow +\,\infty\;\;\;{\rm for\;every}
\;\;\;n>0\;{\rm fixed},
$$
while by Lemma \ref{lem5.1} (see (\ref{CV1})) we have
$$\underset{n\rightarrow +\infty}\limsup\;\underset{k\rightarrow +\infty}\limsup\;\|G_n (w_k)\|_{H^{1}_{0}(\Omega)}=0,$$
and while we have
$$
\underset{n\rightarrow +\infty}\limsup\;\|G_n
(w^\star)\|_{H^{1}_{0}(\Omega)}=0,
$$
since  $w^\star\in H^{1}_{0}(\Omega)$, we conclude that
\begin{equation}\label{5.10000}
w_k \rightarrow w^{\star}\;\;\; {\rm
in}\;\;\;H^{1}_{0}(\Omega)\;\;\;{\rm
strongly\;\;\;as\;\;\;}k\rightarrow +\,\infty,
\end{equation}
which is nothing but (\ref{5.2}).
\\
\\
\noindent{\textbf{Second step}} 
Let us now pass to the limit in
(\ref{care}) as $k$ tends to infinity. This is easy for the first
term of the left-hand side of (\ref{care}) as well as for the three
terms of the right-hand side of (\ref{care}), which pass to the
limit in $(L^{2^\star}\!(\Omega))'$ strongly by a proof which is
similar to the one that we used in  the proof of Lemma \ref{lem
436}.
 \\
 \indent It remains to pass to the limit in the second term of the left-hand side of (\ref{care}), namely in
$$T_k\big(K_{\delta_0}(x,w_{k},Dw_{k})\big)\,{\rm sign}_k(w_k).$$
 \par
We first observe that in view of (\ref{def K2}) and of
$\delta_0\geq\gamma$ (see (\ref{3777bis})), we have
  $$
 |T_k\big(K_{\delta_0}(x,w_{k},Dw_{k})\big)\,{\rm sign}_k(w_k)|\leq|K_{\delta_0}(x,w_{k},Dw_{k})|
 \leq(c_0+\delta_0)\,\|A\|_{\infty}|Dw_k|^2\;\;\;{\rm a.e. \;\;\;in\;\;\;}\Omega,
$$
which implies that the functions
$T_k\big(K_{\delta_0}(x,w_{k},Dw_{k})\big)\,{\rm sign}_k(w_k)$ are
equiintegrable since $Dw_k$ converges strongly to $Dw^\star$ in
$L^2(\Omega)^N$.
\par Extracting if necessary a subsequence, still denoted by $k$, such that
$$
Dw_k\rightarrow Dw^\star\;\;\;{\rm a.e.}\;\;\; {\rm in} \;\;\;
\Omega,
$$
we claim that

\begin {equation}\label{5.302}
T_k\big(K_{\delta_0}(x,w_{k},Dw_{k})\big)\,{\rm
sign}_k(w_k)\rightarrow K_{\delta_0}(x,w^\star,Dw^\star)\, {\rm
sign}(w^\star)\;\;\;{\rm a.e. \;\;\;in\;\;\;}\Omega.
\end{equation}
\par On the first hand we use the first part of (\ref{consconv}),
which asserts that
$$
K_{\delta_0}(x,w_{k},Dw_{k})\rightarrow
K_{\delta_0}(x,w^\star,Dw^\star)\;\;\; {\rm a.e.
\;\;\;in\;\;\;}\Omega,
$$
and the fact that for every $s\in \mathbb{R}$
$$
T_k(s_k)\rightarrow s \;\;\; {\rm if }\;\;\;k\rightarrow
+\infty\;\;\;{\rm when}\;\;\;s_k\rightarrow s,
$$
to deduce that
\begin {equation}\label{5.xxx}
T_k\big(K_{\delta_0}(x,w_{k},Dw_{k})\big) \rightarrow
K_{\delta_0}(x,w^\star,Dw^\star) \;\;\;{\rm a.e.
\;\;\;in\;\;\;}\Omega.
\end{equation}
\par On the other hand we use the fact that
$$
{\rm sign}_k (w_k) \rightarrow {\rm sign} (w^\star)\;\;\; {\rm a.e.
\;\;\;in}\;\;\;\{y\in\Omega\,:\,w^\star(y)\neq 0\},
$$
which together with (\ref{5.xxx}) proves the convergence
(\ref{5.302}) in the set $\{ y\in \Omega\;:\;w^\star(y)\neq 0\}$.\\
\par  Finally, as far as the convergence in the set $\{ y\in \Omega\;:\;w^\star(y)= 0\}$ is concerned, convergence (\ref{5.xxx}),
the fact that (see (\ref{3.18biss}))
$$
K_{\delta_0}(x,w^\star,Dw^\star) = 0\;\;\;{\rm a.e. \;\;\,in}\;\;\;
\{y\in \Omega\,:\,w^\star (y)=0\},
$$
and the fact that 
$|{\rm sign}_k (s)|
\leq 1$
for every $s\in \mathbb{R}$ 
together prove that
$$
T_k\big(K_{\delta_0}(x,w_{k},Dw_{k})\big) \, {\rm sign}_k(w_k)
\rightarrow 0 = K_{\delta_0}(x,w^\star,Dw^\star) \, {\rm
sign}(w^\star) \;\;\;{\rm a.e.\;\;\; in\;\;\;} \{y\in
\Omega\,:\,w^\star (y) =0 \}.
$$
\par This completes the proof of (\ref{5.302}).\\
%
%
\par The equiintegrability and the almost everywhere convergence  of $T_k\big(K_{\delta_0}(x,w_{k},Dw_{k})\big)\, {\rm sign}_k (w_k)$
then imply that
$$
T_k\big(K_{\delta_0}(x,w_{k},Dw_{k})\big)\, {\rm sign}_k
(w_k)\rightarrow K_{\delta_0}(x,w^\star,Dw^\star)\,{\rm sign}
(w^\star)\;\;\;{\rm in}\;\;\;L^1(\Omega)\;\;\;{\rm strongly}.
$$
\par
This proves that $w^\star $ satisfies (\ref{CCbis}). Since $w^\star$
also satisfies (\ref{CCter}) (see (\ref{5.1bis})),
 Theorem \ref{theo34} is proved.
\qquad$\Box$



\section{Appendix}
In this Appendix, we give an estimate of the function $g_{\delta}$
defined by (\ref{def g}) (see Lemma \ref{lemma32}), and the
definitions of the constants $\delta_0$ and $Z_{\delta_0}$ which
appear in Theorem \ref{theo} (see Lemma \ref{lemma 2}).
\subsection{An estimate for the function $\bf{g_{\delta}}$}
\begin{lemma}\label{lemma32}
For $\delta>0$, let $g_\delta:\;\mathbb{R}\rightarrow \mathbb{R}$ be
the function defined by {\rm(\ref{def g})}, i.e. by
\begin{equation}
g_{\delta}(t)=-|t|+\frac{1}{\delta}(1+\delta |t|)\log(1+\delta
|t|),\;\;\;\forall t\in\mathbb{R}.
\end{equation}
\noindent Then, for every $\lambda$ and $\delta_{\star}$ with
\begin{equation}
0<\lambda<1,\quad 0<\delta_{\star}<+\infty,
\end{equation}
there exists a constant $C(\lambda)$ which depends only on
$\lambda$, with
\begin{equation}\label{6.23bis}
0<C(\lambda)\leq \sup\left\{1\,,\,\dfrac{2^{1+\lambda}}{\lambda
e}\right\},
\end{equation}
such that
\begin{equation}\label{aster}
0\leq g_{\delta}(t)\leq
\delta_{\star}^{\lambda}C(\lambda)|t|^{1+\lambda}, \;\;\;\forall
t\in\mathbb{R},\;\;\;\forall \delta,\; 0<\delta\leq
\delta_{\star}.
\end{equation}
Moreover
\begin{equation}\label{6.4bis}
0\leq g_{\delta}(t)<
\delta_{\star}^{\lambda}C(\lambda)|t|^{1+\lambda}, \;\;\;\forall
t\in\mathbb{R},\; t\neq0,\;\;\;\forall \delta,\; 0<\delta\leq
\delta_{\star}.
\end{equation}

\end{lemma}
\begin{proof}
Let $ g: \mathbb{R^+}\rightarrow \mathbb{R}$ be the function defined
by
$$
g(\tau)=-\tau+(1+\tau)\log (1+\tau),\;\;\;\forall \tau \geq 0.
$$
Since $g(0)=0$ and $g'(\tau)\geq 0$, one has
\begin{equation}\label{6A}
g(\tau)\geq 0,\;\;\;\forall \tau\geq 0.
\end{equation}
\indent On the other hand, since $\log(1+\tau)< \tau$ for $\tau>0$, one has
$g(\tau)< \tau^{2}\leq \tau^{1-\lambda}\tau^{1+\lambda}$ for $\tau>0$,  and
therefore for $0<\lambda<1$ and for every $m>0$
\begin{equation}\label{6B}
g(\tau)< m^{1-\lambda}\tau^{1+\lambda},\;\;\;\forall \tau,\; 0< \tau\leq m .
\end{equation}
One has also
$$
\dfrac{g(\tau)}{\tau^{1+\lambda}}<
\dfrac{(1+\tau)\log(1+\tau)}{\tau^{1+\lambda}}=\left(\dfrac{1+\tau}{\tau}\right)^{1+\lambda}
\dfrac{\log(1+\tau)}{(1+\tau)^{\lambda}},\quad\forall \tau> 0,
$$
and therefore
$$
\dfrac{g(\tau)}{\tau^{1+\lambda}}<\left(\dfrac{1+m}{m}\right)^{1+\lambda}
\dfrac{\log(1+\tau)}{(1+\tau)^{\lambda}}\;\;\;\forall \tau\geq m>0.
$$
But the function $\dfrac{\log(1+\tau)}{(1+\tau)^{\lambda}}$ reaches its
maximum for $\tau_{0}$ defined by $(1+\tau_{0})=e^{1/\lambda}$, hence
$$
\dfrac{\log(1+\tau)}{(1+\tau)^{\lambda}}\leq \frac{1}{\lambda
e},\;\;\;\forall \tau\geq 0.
$$
This implies that for $0<\lambda<1$ and for every $m>0$
\begin{equation}\label{6C}
g(\tau)< \left(\dfrac{1+m}{m}\right)^{1+\lambda}\frac{1}{\lambda
e}\;\;\tau^{1+\lambda},\;\;\;\forall \tau\geq m,
\end{equation}
which, with (\ref{6A}) and (\ref{6B}), implies that for
$0<\lambda<1$ and for every $m>0$
\begin{equation}
\displaystyle 0\leq g(\tau)<\sup\left\{
m^{1-\lambda}\,,\,\left(\dfrac{1+m}{m}\right)^{1+\lambda}\dfrac{1}{\lambda
e}\right\}\tau^{1+\lambda},\;\;\;\forall \tau> 0,
\end{equation}
or in other terms that for every $\lambda$, \,$0<\lambda<1$,
\begin{equation}\label{6D}
0 \leq g(\tau)< C(\lambda)\,\tau^{1+\lambda},\;\;\;\forall \tau> 0,
\end{equation}
for some constant $C(\lambda)$, with (take $m=1$)
$$
0<C(\lambda)\leq \sup\left\{1\,,\,\dfrac{2^{1+\lambda}}{\lambda
e}\right\},
$$
which is nothing but (\ref{6.23bis}).
\par Since
 $$g_{\delta}(t)=\dfrac{1}{\delta}g(\delta|t|),\;\;\;\forall t\in \mathbb{R},$$
one deduces from (\ref{6D}) that $g_{\delta}$ satisfies

\begin{equation*}
\left\{\begin{array}{l} \displaystyle 0\leq g_{\delta}(t)\leq
{\delta}^{\lambda} C(\lambda)|t|^{1+\lambda}, \;\;\;\forall t\in
\mathbb{R},\;\;\;\forall \delta>0,
\\
\\
\displaystyle 0\leq g_{\delta}(t)< {\delta}^{\lambda}
C(\lambda)|t|^{1+\lambda},\;\;\;\forall
t\in\mathbb{R},\; t\neq0,\;\;\;\forall \delta>0,
\end{array}
\right.
\end{equation*}
which proves (\ref{aster}) and (\ref{6.4bis}) with a constant
$C(\lambda)$ which
satisfies (\ref{6.23bis}).
\qquad$\Box$
\end{proof}
%
\subsection{Definition of $\delta_0$ and $Z_{\delta_{0}}$}\label{sub52}
The goal of this Subsection is to define the constants $\delta_0$
and $Z_{\delta_0}$ which appear in Theorem \ref{theo}. We will prove
the following result.
%
%
\begin{lemma}\label{lemma 2}
Assume that \rm{(\ref{2.2}), (\ref{2.3}), (\ref{2.6}), (\ref{2.5}),
(\ref{2.5bis}) and  (\ref{2.4})} hold true. Assume moreover that the two smallness conditions
 {\rm(\ref{A1})} \it and
{\rm{(\ref{A3})}} hold true.
\par Let $\delta_1$ be the number defined by
\begin{equation}\label{Ldelta1 bis}
 \delta_{1}=\dfrac{\alpha -C_{N}^2\|a_0\|_{N/2}}{C_{N}^{2}\|f\|_{N/2}}.
 \end{equation}
 One has
 \begin{equation}\label{6.18bis}
\delta_1 > \gamma.
\end{equation}
 \par For $\delta\geq0$, let $\Phi_{\delta}\,:\,\mathbb{R^{+}}\rightarrow\mathbb{R}$ (see Figure 2) be the function defined by
\begin{equation}\label{Phidelta bis}
 \Phi_{\delta} (X)=GC^{2+\theta}_{N}\|a_0\|_q X^{1+\theta}-(\alpha-C_{N}^2\|a_0\|_{N/2}-\delta C_{N}^{2}\|f\|_{N/2})X+\|f\|_{H^{-1}(\Omega)},
 \end{equation}
where $\theta$ is defined by {\rm({\ref{b})}} 
(note that $0 < \theta < 1$ in view of  {\rm({\ref{c})}})
and where $G$ is the constant defined by
\begin{equation}\label{G bis}
  G=\left(\dfrac{\alpha-C_{N}^{2}\|a_0\|_{N/2}}{C_{N}^{2}\|f\|_{N/2}}\right)^{\theta} C(\theta),
 \end{equation}
 with  $C_{N}$ the best constant in the Sobolev's inequality {\rm(\ref{2.7'})} and $C(\theta)$  the constant which appears in {\rm(\ref{aster})}(see also {\rm(\ref{6.23bis})}).
 \\
\indent Then, for $0\leq\delta\leq\delta_1$, the function
$\Phi_\delta$ has a unique minimizer $Z_\delta$ on
$\mathbb{R}^{+}$, which is given by
\begin{equation}\label{Zdelta bis}
 Z_\delta=\left(\frac{ \alpha-C_{N}^2\|a_0\|_{N/2}-\delta C_{N}^{2}\|f\|_{N/2} }{(1+\theta)GC^{2+\theta}_{N}\|a_{0}\|_{q}}\right
 )^{1/\theta}, \;\;\;for\;\;\; 0\leq\delta\leq\delta_1.
 \end{equation}
\indent Moreover, there exists a unique number $\delta_0$ such that
\begin{equation}\label{3777bis}
\gamma\leq\delta_0<\delta_1,
\end{equation}
and
\begin{equation}\label{Phi delta0 bis}
\Phi_{\delta_0}(Z_{\delta_0})=0.
\end{equation}
\\
\indent This number is the number $\delta_0$ which appear in Theorem
\ref{theo}, and $Z_{\delta_0}$ is then defined from $\delta_0$
through formula {\rm(\ref{Zdelta bis})}, namely by
\begin{equation}\label{6.15bis}
 Z_{\delta_0}=\left(\frac{ \alpha-C_{N}^2\|a_0\|_{N/2}-\delta_{0} C_{N}^{2}\|f\|_{N/2} }{(1+\theta)GC^{2+\theta}_{N}\|a_{0}\|_{q}}\right
 )^{1/\theta}.
 \end{equation}
 
\end{lemma}


\begin{remark}\label{rmq6.2bis}
{\rm
Let us explain the meaning of the results stated in Lemma \ref{lemma 2}. 
\par As we will see in the proof of Lemma \ref{lemma 2} (see also Figure 2), 
the function $\Phi_\delta$ is the restriction to $\mathbb{R}^{+}$ of a function which looks like a convex parabola.
This function attains its minimum at a unique point $Z_{\delta}$, and for $\delta$ which satisfies $\delta < \delta_1$ with $\delta_1$ given by (\ref{Ldelta1 bis}), one has $Z_{\delta} > 0$. 
\par The smallness condition (\ref{A1}) is equivalent to the fact that $\delta_1 > \gamma$, 
and the smallness condition (\ref{A3}) to the fact that the minimum $\Phi_\gamma (Z_{\gamma})$ of $\Phi_\gamma$ is nonpositive.
For $\delta = \delta_1$, the minimum $\Phi_{\delta_1} (Z_{\delta_1} )$ of $\Phi_{\delta_1}$ is equal to $\|f\|_{H^{-1}(\Omega)}$, which is strictly positive.
Therefore it can be proved that there exists some $\delta_0$ with 
$\gamma\leq\delta_0<\delta_1$ (see (\ref{3777bis}))
such that the minimum $\Phi_{\delta_0} (Z_{\delta_0})$ of $\Phi_{\delta_0}$ is equal to zero (see~(\ref{Phi delta0 bis})), or in other terms such that the function $\Phi_{\delta_0}$ has a double zero in $Z_{\delta_0}$.
Moreover, when $\gamma < \delta_0$, for every $\delta$ with $\gamma \leq \delta < \delta_0$,
the function $\Phi_{\delta}$ has two distinct zeros $Y_{\delta}^- $ and $Y_{\delta}^+$ with $Y_{\delta}^- < Y_{\delta}^+$ which satisfy 
$0<Y_{\delta}^-<Z_{\delta_0}<Y_{\delta}^+$ (see (\ref{6.25}) in Remark \ref{rmq6.4}).
\qquad $\Box$
} 
\end{remark}


\begin{remark}\label{rmq7}
{\rm
 In the present paper we use Lemma \ref{lemma32} with $\lambda =\theta$ defined by (\ref{b})
 (note that $0<\theta<1$ in view of (\ref{c})) and with $\delta_{\star}=\delta_{1}$  defined
 by (\ref{Ldelta1 bis}). Using the fact that $G$ defined by (\ref{G bis}) is nothing but $G=\delta_{1}^{\theta}C(\theta)$,
 inequalities (\ref{aster}) and (\ref{6.4bis}) imply that

\begin{equation}\label{6.17bis}
\left\{\begin{array}{l}
 0\leq g_{\delta}(t)\leq \delta_{1}^{\theta}
C(\theta)|t|^{1+\theta}=G|t|^{1+\theta},\;\;\; \forall t\in
\mathbb{R},\;\;\; \forall\delta,\;0<\delta\leq\delta_{1},
\\
\\
 0\leq g_{\delta}(t)< \delta_{1}^{\theta}
C(\theta)|t|^{1+\theta}=G|t|^{1+\theta},\;\;\; \forall t\in
\mathbb{R} ,\; t\neq0,\;\;\;
\forall\delta,\;0<\delta\leq\delta_{1}.
\end{array}
\right.
\end{equation}
In particular for $\delta=\delta_0$ defined by (\ref{3777bis}) and
(\ref{Phi delta0 bis}) one has
\begin{equation}\label{6.16}
0\leq g_{\delta_0}(t)\leq G|t|^{1+\theta},\;\;\; \forall t\in
\mathbb{R}.
\end{equation}
} 
\hspace{13cm}$\qquad \Box$
\end{remark}

%
%
\begin{figure}[H]
\begin{center}
\includegraphics[width=7cm,height=4cm]{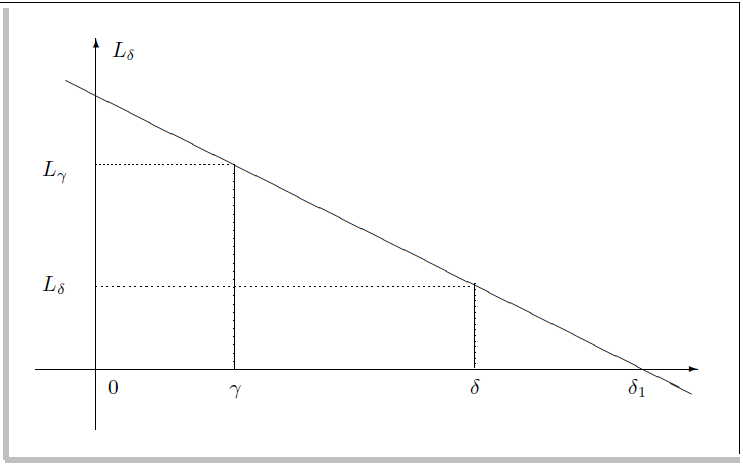}
\end{center}
\caption{The graph of the straight line $L_\delta$ }
\label{fig:dessin1}
\end{figure}

\begin{figure}[H]
\begin{center}
\includegraphics[width=13cm,height=9cm]{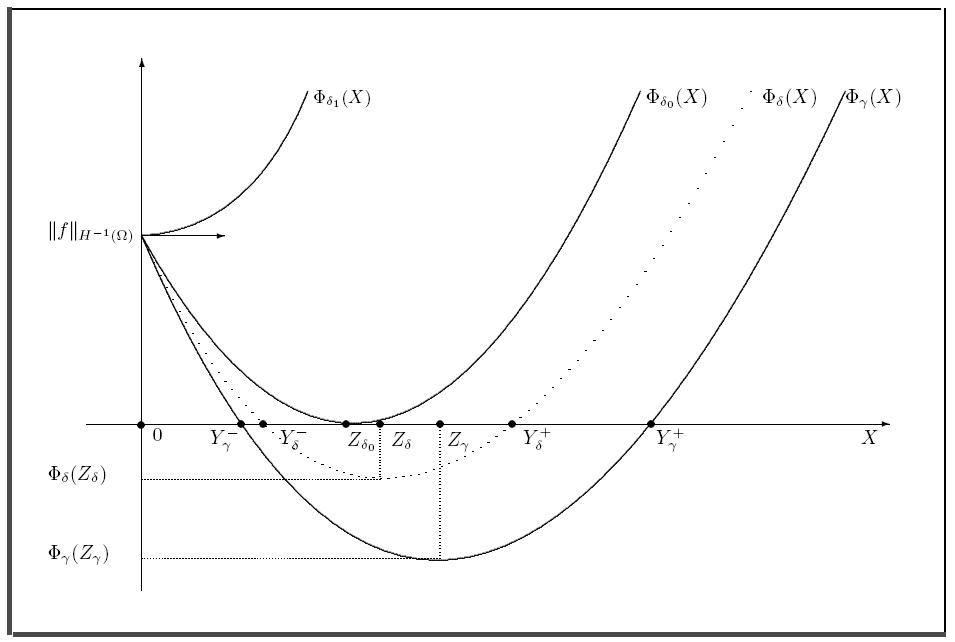}
\end{center}
\caption{The graphs of the functions $\Phi_{\delta}(X)$ for
$\delta=\gamma,\, \gamma<\delta<\delta_{0}$,\,$\delta=\delta_0$  and
$\delta=\delta_{1}$} \label{fig:dessin2}
\end{figure}


\vskip .5cm


\noindent {\bf {Proof of Lemma \ref{lemma 2}}}
For $\delta\geq 0$, let $L_{\delta}$ be the constant defined by (see
Figure 1)
\begin{equation}\label{Ldelta bis}
 L_{\delta}=\alpha-C_{N}^2\|a_0\|_{N/2}-\delta C_{N}^{2}\|f\|_{N/2},
 \end{equation}
 where $C_{N}$ is the best constant in the Sobolev's inequality (\ref{2.7'}). Note that $L_{\delta}$ is decreasing with respect to $\delta$.

\par Since $\delta_1$ is defined by (\ref{Ldelta1 bis}), one has
 $L_{\delta_{1}}=0.$
On the other hand, the first smallness condition (\ref{A1}) is nothing but $L_{\gamma }>0$. Since $L_\delta$ is decreasing in $\delta$, one has $\delta_1 > \gamma$, i.e. 
(\ref{6.18bis}).
 \par
 Let us now study the family of functions $\Phi_{\delta}\,:\,\mathbb{R}^{+}\rightarrow\mathbb{R}$ defined by
(\ref{Phidelta bis}),  i.e., in view of the definition (\ref{Ldelta
bis}) of $L_{\delta}$, by
 \begin{equation}\label{Phidelta}
 \Phi_{\delta} (X)=GC^{2+\theta}_{N}\|a_0\|_q X^{1+\theta}-L_{\delta}X+\|f\|_{H^{-1}(\Omega)},\;\;\; \forall X\geq 0,
 \end{equation}
 (see Figure 2).

  \par Since $a_0\neq 0$ (see (\ref{2.5})), each function $\Phi_{\delta}$ looks like the restriction to $\mathbb{R}^+$ of a convex parabola. 
When $0\leq\delta\leq\delta_1$, one has $L_{\delta}\geq 0$, and this convex parabola has a unique
  minimizer $Z_\delta$ on~$\mathbb{R}^{+}$ which is also the minimizer of the function $\Phi_{\delta}$. A simple computation shows that $Z_{\delta}$ is given by
 \begin{equation}\label{Zdelta}
 Z_\delta=\left(\frac{L_{\delta} }{(1+\theta)GC^{2+\theta}_{N}\|a_{0}\|_{q}}\right )^{1/\theta}=\left(\frac{ \alpha-C_{N}^2\|a_0\|_{N/2}-\delta C_{N}^{2}\|f\|_{N/2} }{(1+\theta)GC^{2+\theta}_{N}\|a_{0}\|_{q}}\right )^{1/\theta},
 \end{equation}
 i.e   (\ref{Zdelta bis}), and that the minimum of  $\Phi_{\delta}$, namely $\Phi_{\delta}(Z_{\delta})$, is given by
 \begin{equation}\label{phiz}
 \left\{
 \begin{array}{l}
   \Phi_{\delta}(Z_{\delta})=\|f\|_{H^{-1}(\Omega)}-\dfrac{\theta}{1+\theta}\, \dfrac{L_{\delta}^{\;\,(1+\theta) /\theta}}{((1+\theta)GC_{N}^{2+\theta}\|a_{0}\|_{q})^{1/\theta}} =\\
   \\
  = \|f\|_{H^{-1}(\Omega)}-\dfrac{\theta}{1+\theta} \dfrac{(\alpha-C_{N}^2\|a_0\|_{N/2}-\delta C_{N}^{2}\|f\|_{N/2} )^{(1+\theta)/\theta}}{((1+\theta)GC^{2+\theta}_{N}\|a_{0}\|_{q} )^{1 /\theta}}.
 \end{array}
 \right.
 \end{equation}
 \\
\indent When $0\leq\delta\leq\delta_1$, the function $L_{\delta}$ is
nonnegative, continuous and decreasing with respect to~$\delta$.
Therefore $Z_{\delta}$ is continuous and decreasing with respect to
$\delta$, while $\Phi_{\delta}(Z_{\delta})$ is continuous and
increasing with respect to $\delta$.
\\
\indent When $\delta=\delta_1$, one has  $L_{\delta_{1}}=0$, the
function $\Phi_{\delta_1}$ attains its minimum in $Z_{\delta_{1}}=0$, 
and $\Phi_{\delta_1}(Z_{\delta_1})=\break
=\|f\|_{H^{-1}(\Omega)}>0$, 
while the second  smallness condition (\ref{A3}) is nothing
but $\Phi_{\gamma}(Z_{\gamma})\leq0$. Therefore there exists  a
unique $\delta_0$ with $\gamma\leq\delta_0<\delta_1$ such that
$\Phi_{\delta_0}(Z_{\delta_0})=0$. This is the definition of
$\delta_0$ given by (\ref{3777bis}) and  (\ref{Phi delta0 bis})
 in Lemma \ref{lemma 2}.
\\
\indent Lemma \ref{lemma 2} is proved.
\qquad $\Box$

\begin{remark}\label{rmq6.4}
{\rm The case where equality takes places in inequality  (\ref{A3})
corresponds to the case where $\delta_0=\gamma$.\\
\indent On the other hand, 
when (\ref{A3}) is a strict inequality, 
one has  $\gamma < \delta_0$,
and for $\delta$ with
$\gamma \leq \delta < \delta_0$, the function $\Phi_\delta$ has two distinct
zeros $Y_{\delta}^-$ and $Y_{\delta}^+$ with
$0<Y_{\delta}^-<Y_{\delta}^+$. Since 
\begin{equation*}
 \left\{
 \begin{array}{l}
\Phi_{\delta} (X)=GC^{2+\theta}_{N}\|a_0\|_q X^{1+\theta}-(\alpha-C_{N}^2\|a_0\|_{N/2}-\delta C_{N}^{2}\|f\|_{N/2})X+\|f\|_{H^{-1}(\Omega)} = 
\\
\\
= \Phi_{0} (X) + \delta C_{N}^{2}\|f\|_{N/2} X, \end{array}
 \right.
 \end{equation*}
the family of functions
$\Phi_\delta$ is an increasing family of functions on $\mathbb{R^{+}}$, and one has
\begin{equation}\label{6.25}
0<Y_{\delta}^-<Z_{\delta_0}<Y_{\delta}^+\;\;\; {\rm if}\;\;\;
\gamma \leq \delta < \delta_0.
\end{equation}
\hspace{13.9cm} \qquad $\Box$
}
\end{remark}

\vskip 1cm

\section*{Acknowledgments}
This work has been done during visits of the first author to the
Laboratoire Jacques-Louis\break Lions, Universit\'{e} Pierre et Marie
Curie (Paris VI) et CNRS, Paris, 
and of the second author to the\break D\'{e}partement de Math\'{e}matiques de l'Ecole Normale
Sup\'{e}rieure, Vieux Kouba, Alger.
Both authors wish to thank the Projet Tassili 08
MDU 736, the D\'{e}partement de Math\'{e}matiques de l'Ecole Normale
Sup\'{e}rieure and the Laboratoire
Jacques-Louis Lions whose supports made possible these visits and this work.

\vskip 2cm


\end{document}